\documentclass[11pt,leqno]{article}
\include{psfig}
\include{epsf}

\usepackage[utf8]{inputenc} 

\usepackage{color}
\usepackage{amssymb} %to define $\Box$ in latex2e

\newtheorem{theorem}{Theorem}[section]

\newtheorem{lemma}{Lemma}[section]

\newtheorem{definition}{Definition}[section]

\newtheorem{remark}{Remark}[section]
\newtheorem{example}{Example}[section]

\newtheorem{simulation algorithm}{Simulation Procedure}[section]

\catcode`\@=11
\renewcommand{\section}{
         \setcounter{equation}{0}
         \@startsection {section}{1}{\z@}{-3.5ex plus -1ex minus
         -.2ex}{2.3ex plus .2ex}{\normalsize\bf}
}
\renewcommand{\subsection}{
         \@startsection {subsection}{1}{\z@}{-3.5ex plus -1ex minus
         -.2ex}{2.3ex plus .2ex}{\normalsize\bf}%{\center\normalsize\bf}
}
\catcode`\@=12

\def\reals{{\rm\vrule depth0ex width.4pt\kern-.08em R}}
\def\bbbz{{\mathchoice {\hbox{$\sf\textstyle Z\kern-0.4em Z$}}
{\hbox{$\sf\textstyle Z\kern-0.4em Z$}}
{\hbox{$\sf\scriptstyle Z\kern-0.3em Z$}}
{\hbox{$\sf\scriptscriptstyle Z\kern-0.2em Z$}}}}

\newcommand{\nc}{\newcommand}
\nc{\W}{{\bf W}}
\nc{\A}{{\bf A}}
\nc{\bL}{{\bf L}}
\nc{\bH}{{\bf H}}
\nc{\C}{{\cal C}}

\def\eq#1{(\ref{e:#1})}

\def\elabel#1{\label{e:#1}}

\begin{document}
%\begin{titlepage}
\begin{center}
\Large\bf Internet of quantum blockchains: security modeling and dynamic resource pricing for stable digital currency
%:Fairness and Pareto Optimality
%Pareto Optimal Nash Equilibrium
\end{center}
\begin{center}
\large\bf Wanyang Dai~\footnote{The project is funded by National Natural Science Foundation of China with Grant No. 11771006, Grant No. 10971249, and Grant No. 11371010. This paper initially appeared in Preprints (called Proceeding) of 22th Annual Conference of Jiangsu Association of Applied Statistics, pages 7-45, November 13-15, 2020, Suzhou, China.}
\end{center}
\begin{center}
\small Department of Mathematics\\
and State Key Laboratory of Novel Software Technology\\
Nanjing University, Nanjing 210093, China\\
Email: nan5lu8@nju.edu.cn\\
25 April 2021
\end{center}

\vskip 0.1 in
\begin{abstract}
Internet of quantum blockchains (IoB) will be the future Internet. In this paper, we make two new contributions to IoB: developing a block based quantum channel networking technology to handle its security modeling in face of the quantum supremacy and establishing IoB based FinTech platform model with dynamic pricing for stable digital currency. The interaction between our new contributions is also addressed. In doing so, we establish a generalized IoB security model by quantum channel networking in terms of both time and space quantum entanglements with quantum key distribution (QKD). Our IoB can interact with general structured things (e.g., supply chain systems) having online trading and payment capability via stable digital currency and can handle vector-valued data streams requiring synchronized services. Thus, within our designed QKD, a generalized random number generator for private and public keys is proposed by a mixed zero-sum and non-zero-sum resource-competition pricing policy. The effectiveness of this policy is justified by diffusion modeling with approximation theory and numerical implementations.\\

%\noindent {\em Subject classifications:}
\end{abstract}
%\end{titlepage}

\section{Introduction}

Internet of quantum blockchains (IoB) will be the future Internet (see, e.g., Dai~\cite{dai:plamod,dai:quacom}, Rajan and Visser~\cite{rajvis:quablo}, SIR Forum~\cite{sirfor}). Especially, with the coming of big data era (in terms of data, computing power, and algorithms) and the quick evolution from the currently implementing Industrial 4.0 (IR 4.0, see, e.g., Schwab~\cite{sch:fouind}) to IR 6.0 (see, e.g., SIR Forum~\cite{sirfor}: The Sixth Industrial Revolution Forum), blockchain and quantum computing will be the core technology of IR 6.0 (see, e.g., Courtland~\cite{cou:chiqua}, Dai~\cite{dai:plamod,dai:quacom}, Deutsch~\cite{deu:quacom}, Feynman~\cite{fey:quamec}, Gibney~\cite{gib:chisat}, Nielsen and Chuang~\cite{niechu:quacom}, SIR Forum~\cite{sirfor}). Comparing with the nowadays' high-performance computing facilities, the future quantum computing system will have extremely powerful processing, storage, tracing, and management capability (see, e.g., the latest development and quantum supremacy of quantum computers in Arute {\it et al.}~\cite{aruary:quasup} and Figure~\ref{quantumsupremacy} (enhanced from Dai~\cite{dai:quablo} for public availability)).
%\begin{figure}[tbh]
%%\centerline{\epsfxsize=7.0in\epsfbox{QuantumSupremacy.eps}}
%\centerline{\epsfxsize=6.0in\epsfbox{QuantumSupremacy.eps}}
%%\centerline{\epsfxsize=3.5in\epsfbox{blockchainmanager.eps}}
%\caption{\footnotesize The latest development and quantum supremacy of quantum computers{\textcolor{red}{, where, QC means quantum computing, CC means %classic computing, CPU means central processing unit, and GPU means graphic processing unit.}}}
%\label{quantumsupremacy}
%\end{figure}
Furthermore, with the effective designs of quantum computing chips and algorithms (see, e.g., Dai~\cite{dai:quacom}, Deutsch~\cite{deu:quacom}, Feynman~\cite{fey:quamec}, Nielsen and Chuang~\cite{niechu:quacom}), the powerful quantum computers are recently  announced available by IBM, Google, Rigetti, etc. These quantum computers make it possible for us to realize the future generation of Internet called IoB (Internet of (quantum) blockchains), i.e., to realize the design of quantum cloud-computing based future Internet with (quantum) blockchain communication protocol in around 25 years as firstly claimed in Dai~\cite{dai:quacom} with the support of Arute {\it et al.}~\cite{aruary:quasup}, etc..

%As investigated in Dai~\cite{dai:plamod,dai:quacom}, (quantum) blockchain is originated from IP and will be the future Internet (IoB) communication %protocol. Starting from the U.S. based ARRP (Advanced Research Projets Agency), IP was initially deployed in U.S. Navy around the late 1960s as %summarized in Comer~\cite{com:inttcp}.

In history, owing to the capacity limitations with respect to processing and storage devices,
%as shown in Figure~\ref{P2P-IP.eps},
the Internet IP protocol is a simple path routing algorithm oriented one and the transmitted IP packet in each node is not (or not fully) stored as summarized in Comer~\cite{com:inttcp}.
%\begin{figure}[tbh]
%%\centerline{\epsfxsize=7.2in\epsfbox{P2P-IP.eps}}
%\centerline{\epsfxsize=6.0in\epsfbox{P2P-IP.eps}}
%\caption{\footnotesize P2P $\&$ IP to blockchain communication protocol}
%\label{P2P-IP.eps}
%\end{figure}
However, due to the rapid increasing of capacity, the blockchain technology with more powerful (e.g., utility-maximization or hash function based) strategy planning capability and data storage at each linked node is quickly deploying while it meets the stringent quality of service (QoS) requirements in various real-world applications (e.g., online communication services and payments with dynamic pricing via stable digital currency).
%Furthermore, this technology is even foreseen its usage in the U.S. Navy Aegis communication system with the encryption of 256 bit hashes (see, e.g., %Figure~\ref{P2P-IP.eps}). In addition, due to the the new quantum-computing Moore's law (see, e.g., Bertels~\cite{ber:quacom} and %Dai~\cite{dai:quacom}) and the realizations of quantum computers with quantum supremacy as shown in Figure~\ref{quantumsupremacy}, we foresee that the %current Internet will be replaced by quantum blockchain based IoB in around 25 years. This IoB will be powered by quantum-cloud-computing with the storage capability in each service center and switch node as proposed in Dai~\cite{dai:plamod,dai:quacom} and Figure~\ref{P2P-IP.eps} (see also the US White House 2020 QNSP plan and China 2020 IoB White Paper).
In other words, the more capable blockchain protocol based on various smart contracts and intelligent engines can be efficiently implemented over the future IoB.
%These smart contracts and intelligent engines include those in terms of resource allocations as studied in Dai~\cite{dai:optrat,dai:plamod,dai:quacom} and those appeared in supply chain finance as shown in Figure~\ref{three2multi}.
%\begin{figure}[tbh]
%%\centerline{\epsfxsize=7.2in\epsfbox{three2multi.eps}}
%\centerline{\epsfxsize=6.0in\epsfbox{three2multi.eps}}
%\caption{\footnotesize Smart contract within blockchain}
%\label{three2multi}
%\end{figure}
%The smart contract processing in a supply chain finance system may be a multi-parties involved one and is a generalization of technologies for voice/video conferences and Bitcoin/Ethereum with decision recording storages (see, e.g., Figures~\ref{P2P-IP.eps}-\ref{three2multi}, Buterin~\cite{but:ethnex}, Nakomoto~\cite{nak:bitpee}). When various complex online payments and transactions in a supply chain are involved as shown in Figure~\ref{SupplyChainFinance}, the stable digital currency and its dynamical pricing within a FinTech platform or a generalized IoT system will be concerned, which will be our focal point in Section~\ref{fodpsdc} of this paper.

Therefore, both US White House 2020 Quantum Network Strategy Plan (QNSP) and China 2020 IoB white paper accompanying with the so-called ``New Seven Capital Construction (NSCC)" (see more insights/designs in Dai~\cite{dai:bloind} and English summary in Morgan Staley's Report by Xing {\em et al.}~\cite{xin:newinf}) are announced in the early year of 2020. However, this powerfulness introduces severe security issue to the currently implementing blockchain as shown in Figure~\ref{quantumsupremacy}. Hence, in this paper, we develop new block based quantum channel networking technology together with the existing discussions (see, e.g., Bennett and Brassard~\cite{benbra:quacry}, Dai~\cite{dai:plamod,dai:quacom}, Rajan and Visser~\cite{rajvis:quablo}, SIR Forum~\cite{sirfor}, Yin {\it et al.}~\cite{yinli:entsec}) to handle the IoB security modeling issue. In this way, our IoB can be claimed as a decentralized (or partially decentralized) security (or secret) union system among users through quantum encryption with quantum key distribution (QKD) and in terms of both time and space quantum entanglements.

Furthermore, in the future more secured IoB, quantum cloud-computing centers will have strong artificial intelligence (AI) to support the future generalized Internet of Things (IoT) or FinTech platform for big data streams and NSCC with dynamic demand-side and supply-side of complex economic structures (e.g., a power and energy grid and those systems in Buterin~\cite{but:ethnex}, Dai~\cite{dai:plamod,dai:quacom}, %Paraskova~\cite{par:oilcoi}, 
and Xing {\em et al.}~\cite{xin:newinf}). More precisely, in a generalized IoT or a FinTech platform in support of digital economy, we try to realize the blockchain based smart contracts and online dynamic pricing with authentication and registration for users and resources through stable digital currency as shown in Figures~\ref{blockchain-fintech}-\ref{blockchaincurrency}.
\begin{figure}[tbh]
%\centerline{\epsfxsize=7.0in\epsfbox{Blockchain-FinTech.eps}}
\centerline{\epsfxsize=5.6in\epsfbox{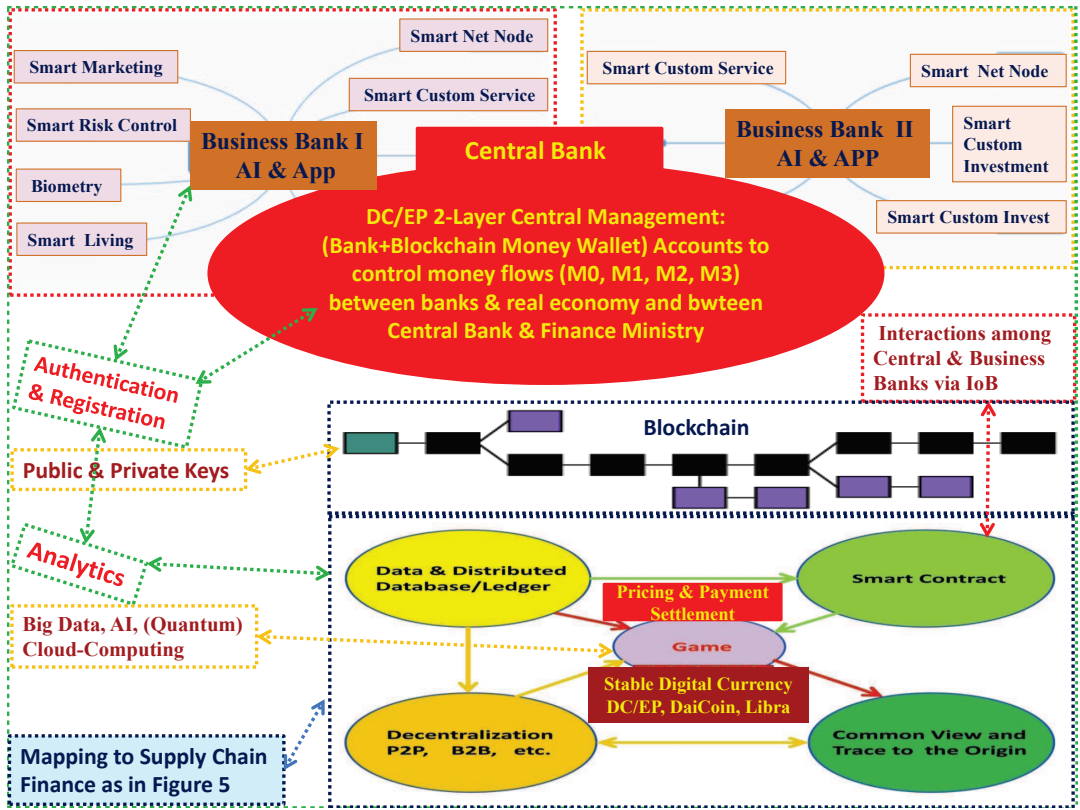}}
\caption{\footnotesize FinTech platform for central $\&$ business banks. The upper-half of this figure presents the two-layer design with three functionalities: authentication, registration, and analytics, which can be implemented via IoB as in the lower-right graph of this figure. The first two functionalities can be done via public and private keys in IoB. The third one can be realized via (quantum) cloud-computing with smart contracts, etc. Each business bank offers services to customers (e.g., the supply chain financing service in Figure 2). DC/EP means digital currency/electronic payment and M0-M3 are the four types of currency supplies. DaiCoin is the first ever designed stable digital currency as illustrated in Figure 5.}
\label{blockchain-fintech}
\end{figure}
\begin{figure}[tbh]
%\centerline{\epsfxsize=7.0in\epsfbox{AI-BigData-Blockchain.eps}}
\centerline{\epsfxsize=6.0in\epsfbox{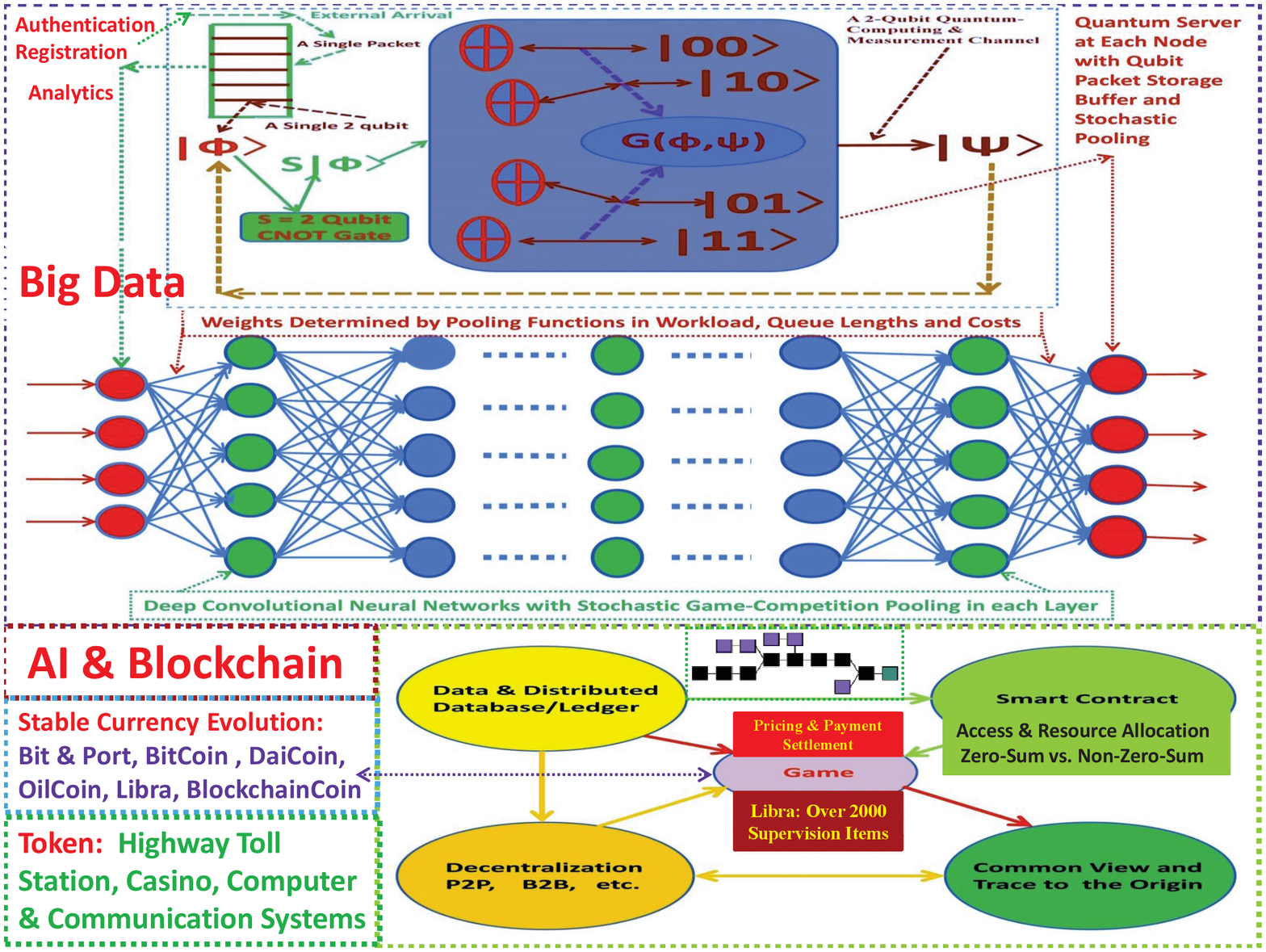}}
\caption{\footnotesize Token and stable digital currency with AI and blockchain}
\label{blockchaincurrency}
\end{figure}
Therefore, around the world and within China, various business and central banks are quickly announcing their proposals and implementations concerning AI and blockchain based FinTech to support the real-world and digital economy (see, e.g., Figure~\ref{blockchain-fintech} and Dai~\cite{dai:inseff} for an illustration). However, there are many issues to be solved concerning how to efficiently implement the business models and technical proposals over IoB. For example, the security issue for authentication and registration. Furthermore, due to the direct interactions between real economy and business banks with associated tax and digital tax issues, the disputes between a nation's central bank and his finance ministry (see, e.g., Figure~\ref{blockchain-fintech}) and even among countries come up. Therefore, just like Libra or digital US dollar, the central bank in China announces his stable digital blockchain currency plan to control the money flows dominated by digital currencies currently used by various business banks. In this digital currency plan, both central bank issuing storehouse and business bank storehouse are designed with authentication, registration, and big data analytics capability as shown in Figure~\ref{blockchain-fintech}-\ref{blockchaincurrency}. Furthermore, this blockchain currency plan is also very helpful to the widely concerned supply chain finance (see, Figure~\ref{supplychainfinance} for an illustration).
%and the quickly evolving health-care system (see, Figure~\ref{MedicalDiagnosis} for illustrations).
In these systems, the cash payments are frequently impossible (e.g., due to production time delays/production lead times in MTO systems)
%or during pandemic/epidemic periods)
and digital currency can be used to replace bank notes, e-bill, etc. as online payments. Thus, how to conduct dynamic pricing for stable digital currency in support of online payments and transactions for digital economy in different scenarios will be an important issue. It will be the focal point of our study in this paper. This newly added random dynamic price information together with resource allocation information and node (bank, cloud service center, and even CPU) information can be used to generate nonce values and private quantum keys through certain utility (hash) functions just like a random number generator in an existing blockchain system (See, e.g., Buterin~\cite{but:ethnex}, Dai~\cite{dai:plamod,dai:quacom}, Iansiti and Lakehani~\cite{ianlak:trublo}, Nakamoto~\cite{nak:bitpee}, Rajan and Visser~\cite{rajvis:quablo}).

%{\textcolor{red}{(Add something here ???), Central Bank of China control M0, etc., Tax. Digital Tax $\&$ Tariff unfair claimed by different %countries %like US, Turkey, and France, e.g., wine, 1/8/2019}}
%\begin{figure}[tbh]
%%\centerline{\epsfxsize=7.0in\epsfbox{SupplyChainFinance.eps}}
%\centerline{\epsfxsize=6.0in\epsfbox{SupplyChainFinance.eps}}
%\caption{\footnotesize Supply chain finance and communication link: a comparison. In this figure, ATO means assemble to order, MTO means make to %order, DD means digital dollar, DC/EP means (China Central Bank) digital currency/electrnic payment, TCP means transmission control protocol, IoB %means Internet of Blockchains.}
%\label{SupplyChainFinance}
%\end{figure}
%\begin{figure}[tbh]
%%\centerline{\epsfxsize=7.0in\epsfbox{MedicalDiagnosis.eps}}
%\centerline{\epsfxsize=6.0in\epsfbox{MedicalDiagnosis.eps}}
%\caption{\footnotesize FinTech and the trend of medical diagnosis. This figure is slightly modified from the ppt slice of Wanyang Dai's plenary talk %in Biomedical Engineering-2019 by adding the words ``Pandemic" and ``covid-19".}
%\label{MedicalDiagnosis}
%\end{figure}

Therefore, due to the involvement of dynamic pricing activities, we consider the interaction between our IoB and general structured things (e.g., supply chain, energy, and health-care systems) as a general FinTech platform model in support of digital economy with online trading and payment capability through stable digital currency while supporting online resource scheduling. In realize the interaction, we propose a resource-competition oriented dynamic pricing and scheduling policy in the supply-side. This policy consists of three stages: In the first stage, we compute the decision of zero-sum game-competition based dynamic users' selection among all the users; In the second and third stages, we derive the decision of non-zero-sum game-competition based resource-sharing among selected users while conducting dynamic pricing. The effectiveness of our designed policy is justified by diffusion modeling with approximation theory and numerical implementations.

To characterize the internal data flow fluctuations of our IoB based FinTech (or IoT) system (especially for vector-valued data of multiple characteristic indices requiring synchronized online service, see, e.g., Skianis {\it et al.}~\cite{skikon:measta},  Toffano and Dubois~\cite{tofdub:quaeig}), we use the triply stochastic renewal reward process (TSRRP) to model the random dynamics of the input quantum qubit data packet flows called big data flows (see, e.g., Dedi\'c and Stanier~\cite{dedsta:towdif}, Dai~\cite{dai:plamod,dai:quacom}, Snijders{\em et al.}~\cite{snimat:bigdat}) in the demand-side. Furthermore, we model the service rate capacity available for resource-competing users at each pool as a randomly capacity region evolving with a finite state continuous Markov chain (FS-CTMC). The parallel-queues in this FinTech platform model are used to storage and buffer quantum qubit data packets from their corresponding users. Each queue may be served at the same time through multiple smart quantum-computing service pools while each pool may also serve multiple queues simultaneously by running intelligent policies in the blockchain. Nevertheless, to reflect the dynamic evolving nature of real-world systems and to realize the decentralized operation in a blockchain, the users to be selected at a time is random, the number of pools to serve a specific queue is random, and the number of queues to be served by a given pool is also random. The efficiency or optimization concerning our proposed policy is in terms of revenue, profit, cost, system delay, etc. We model them through some utility (or hash) functions in terms of the performance measures of their internal quantum qubit data flow dynamics such as queue length and workload processes. To demonstrate the usefulness of our policy, we derive a reflecting diffusion with regime-switching (RDRS) model for the performance measures under our designed policy to offer services to different users in a cost-effective, efficient, and fair way. Based on this RDRS model, our proposed policy is effectively implemented with numerical simulations.

The remainder of this paper is organized as follows. In Section~\ref{secufin}, we develop techniques to deal with IoB security modeling in face of the quantum supremacy and introduce stable digital currency with dynamic pricing. In Section~\ref{fodpsdc}, we formulate system model for dynamic resource pricing of stable digital currency with the required primitives. In Section~\ref{pmodeling}, we present the RDRS model for system performance modeling and main theorem based on a three-stage (i.e., users-selection, dynamic pricing, and resource-competition scheduling) policy. Simulation case studies to show the effectiveness of our policy are also given in this section. In Section~\ref{rdrsm}, we theoretically prove our main theorem. In Section~\ref{cremark}, we give the conclusion of this paper.

\section{IoB Security modeling and stable digital currency}\label{secufin}
%\section{The evolution of IoB and its security}\label{secufin}

In this section, we first develop a block based quantum channel networking technology for IoB security modeling corresponding to the security problem raised by SIR Forum~\cite{sirfor} and displayed in Figure~\ref{quantumsupremacy}
\begin{figure}[tbh]
%\centerline{\epsfxsize=7.0in\epsfbox{QuantumSupremacy.eps}}
\centerline{\epsfxsize=5.6in\epsfbox{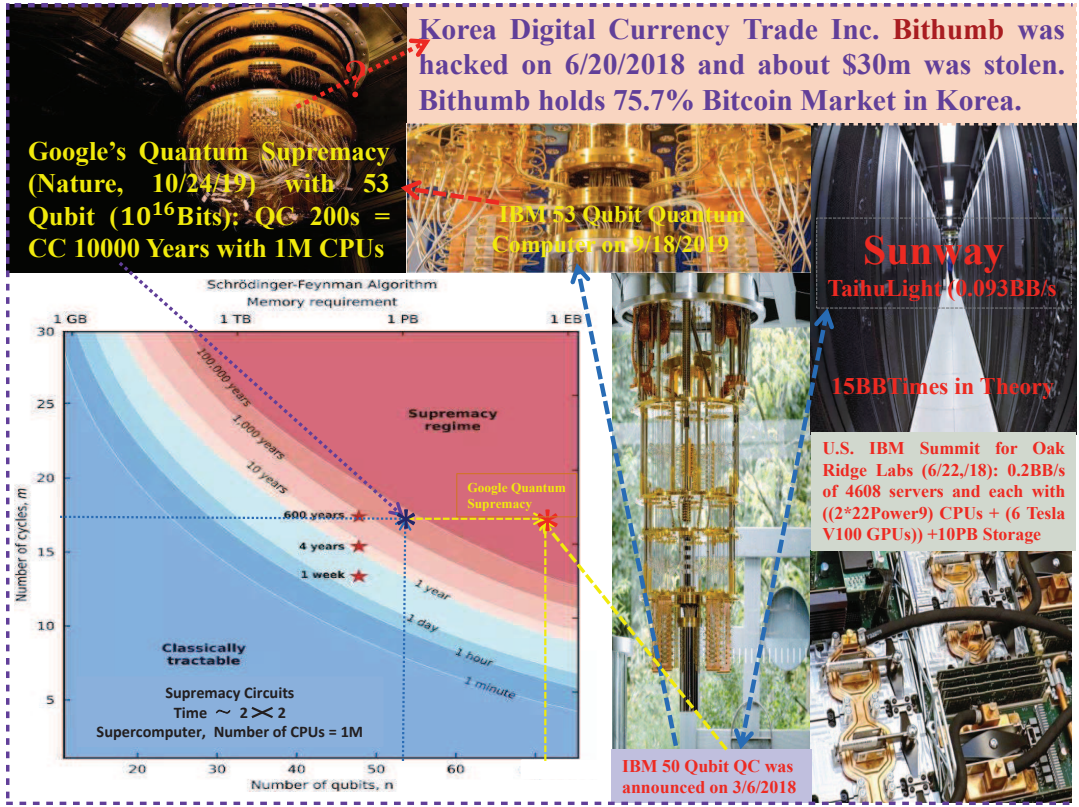}}
%\centerline{\epsfxsize=3.5in\epsfbox{blockchainmanager.eps}}
\caption{\footnotesize Blockchain security and quantum supremacy of quantum computers, where, QC means quantum computing, CC means classic computing, CPU means central processing unit, GPU means graphic processing unit, and BB means billion billion. Furthermore, the currently available China Sunway CC computer and US IBM Summit computer pictured in the lower-right graphs are for the comparison purpose. The Schr$\ddot{o}$dinger-Feynman algorithm memory requirement pictured in the lower-left graph is for the illustration purpose and the red asterisk is the Google quantum supremacy point.}
\label{quantumsupremacy}
\end{figure}
with related illustration in Subsection~\ref{semodiob}. Then, we introduce the concept of stable digital currency as a preparation for its more involved dynamic pricing modeling problem (that will be presented in the next section due to its length).

\subsection{IoB Security modeling}\label{semodiob}

Theoretically, quantum computer can break up any cryptography code used in the existing blockchain systems due to its high-performance computational power. For example, the claimed Google's quantum supremacy with 53 qubit ($10^{16}$ bits): 200 second quantum computing power is approximately 10000 year classic computing power as shown in Arute {\it et al.}~\cite{aruary:quasup} and Figure~\ref{quantumsupremacy}. This quantum supremacy causes the wide concern about the security of blockchain and people wish to have a solution to solve the hacking issue as happened in Korea Digital Currency Trade Inc. (see, e.g., Figure~\ref{quantumsupremacy}). Hence, in this subsection, we develop a generalized IoB security model via a method of block based quantum channel networking.
\begin{figure}[tbh]
%\centerline{\epsfxsize=7.2in\epsfbox{QuantumBlockchainSecurity.eps}}
\centerline{\epsfxsize=5.6in\epsfbox{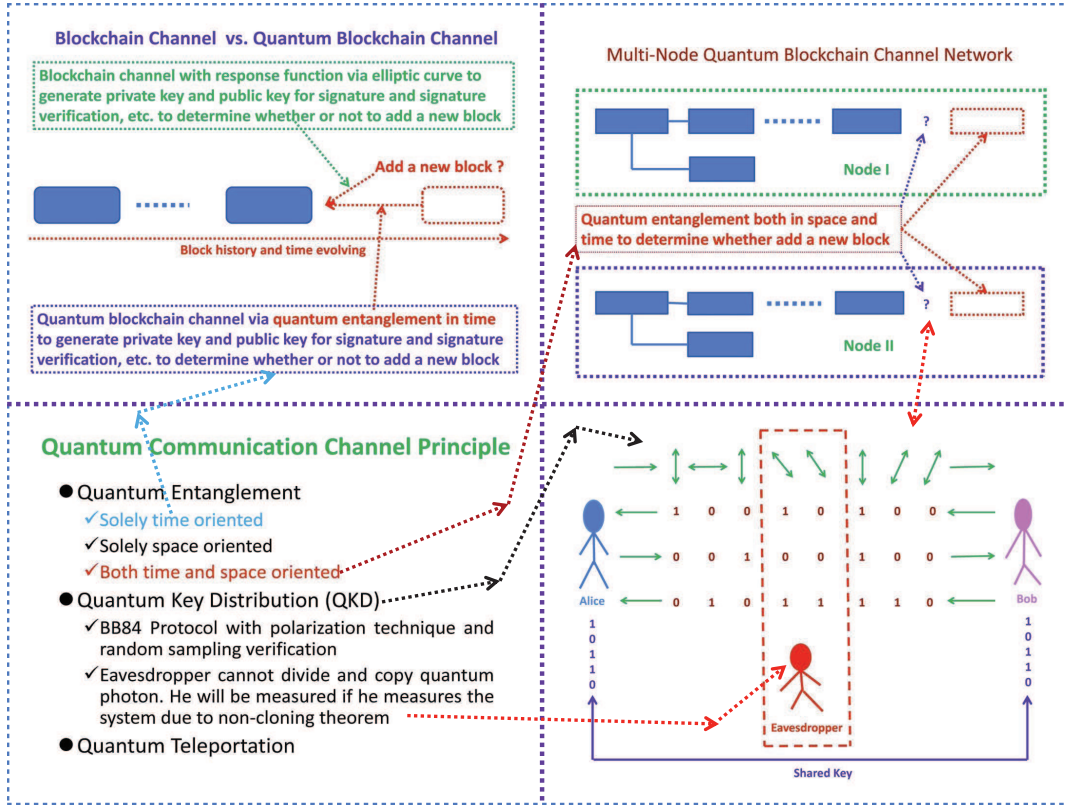}}
\caption{\footnotesize IoB Security modeling based on each quantum channel security. The upper-left graph displays a quantum channel and entanglement between the target future block and its history blocks in time. The upper-right graph displays an added quantum channel between two nodes where the associated quantum entanglements can be in both time and space. The key security techniques for each quantum channel are listed in the two lower graphs and titled as ``Quantum Communication Channel Principal".}
\label{blockchainsecurity}
\end{figure}

More precisely, we set up an IoB security model by networking various quantum channels among blockchain blocks. In doing so, the key point is to establish the corresponding quantum channel between any two blockchain blocks within the same IoB physical node or in different IoB physical nodes while to develop the more safe security model for an quantum channel. The dynamical relationship of input and output qubit data flows over a quantum channel can be modeled through a quantum transfer function (or called a quantum Hash function).
%These blocks can be in different physical network nodes or in a same physical network node.
If two blocks are in two different nodes, the corresponding quantum channel is space-oriented. If two blocks are in the same node and are consecutively ordered in time, the corresponding quantum channel is time-oriented. In this way, a generalized IoB security model can be established by networking various quantum channels via embedding them into the IoB coupled hardware and software system. In implementing this IoB security model, we need to replace the classic cryptographic Hash functions by the corresponding quantum cryptographic Hash functions. Then, by classifying these quantum hash functions into the generalized utility functions and conducting dynamic resource pricing, we can map our newly derived pricing decision information with users' information, nodes' information, etc. into our IoB based public and private quantum keys for being added blocks among nodes. Note that, the dynamic pricing decision information here acts as the generalized nonce value in our quantum blockchain. It is generated by our newly designed decision-making algorithm with the functionality as the traditional random number generator used in a classic blockchain (see, e.g., Buterin~\cite{but:ethnex} and Nakamoto~\cite{nak:bitpee}). This new algorithm with IoB performance modeling justification will be presented in the subsequent sections due to its length. Furthermore, unlike the classic cryptography or even existing quantum cryptography, we here need to integrate more quantum mechanical technologies into our newly designed IoB security model (especially in the channel security level). These technologies include the generalized quantum channel modeling formula, the space-time quantum entanglement and the quantum key distribution (QKD) protocols (e.g., BB84 protocol) with polarization scheme and random sampling verification for quantum channel security (see, e.g., Dai~\cite{dai:quacom}, Yin {\it et al.}~\cite{yinli:entsec}, Rajan and Visser~\cite{rajvis:quablo}, Bennett and Brassard~\cite{benbra:quacry}, %Bernstein {\it et al}~\cite{berbuc:posqua},
Lange and Steinwandt~\cite{lanste:posqua}).

Concerning this point, our security model for each quantum channel within our IoB can be developed as shown in Figure~\ref{blockchainsecurity} (enhanced from Dai~\cite{dai:secqua} for public availability).  Since the quantum channels can be established inside a single node or among different nodes, we can conduct quantum entanglements with respect to time-indexed history blocks inside a node as shown in the upper-left graph of Figure~\ref{blockchainsecurity} or with respect to space-indexed position blocks among nodes as shown in upper-right graph of Figure~\ref{blockchainsecurity} to decide whether add new blocks or not. Even more, based on the non-cloning theorem (see, e.g., Niestegge~\cite{nie:noncon}, Wootters and Zurek~\cite{woozur:sinqua}), the well-known quantum key distribution (QKD) BB84 protocol with polarization technique and random sampling verification (see, e.g., Bennett and Brassard~\cite{benbra:quacry}) can be incorporated into our time and space based quantum channels as shown in the two lower graphs of Figure~\ref{blockchainsecurity}.

\subsection{Stable digital currency with dynamic pricing}

%The smart contract processing in a supply chain finance system may be a multi-parties involved one and is a generalization of technologies for voice/video conferences and Bitcoin/Ethereum with decision recording storages (see, e.g., Figures~\ref{P2P-IP.eps}-\ref{three2multi}, Buterin~\cite{but:ethnex}, Nakomoto~\cite{nak:bitpee}). When various complex online payments and transactions in a supply chain are involved as shown in Figure~\ref{SupplyChainFinance}, the stable digital currency and its dynamical pricing within a FinTech platform or a generalized IoT system will be concerned, which will be our focal point in Section~\ref{fodpsdc} of this paper.

In this subsection, we first introduce the concept of stable digital currency and then present the real-world practice of dynamic pricing through stable digital currencies by designing two supply chains as shown in Figures~\ref{daicoin.eps}-\ref{supplychainfinance}. In the next section, we will establish general dynamic model with resource pricing decisions that can be used for handling these two systems.

Stable digital currency is a digital token used in digital informational and data network systems.
\begin{figure}[tbh]
%\centerline{\epsfxsize=7.2in\epsfbox{DaiCoinInsur.eps}}
\centerline{\epsfxsize=5.6in\epsfbox{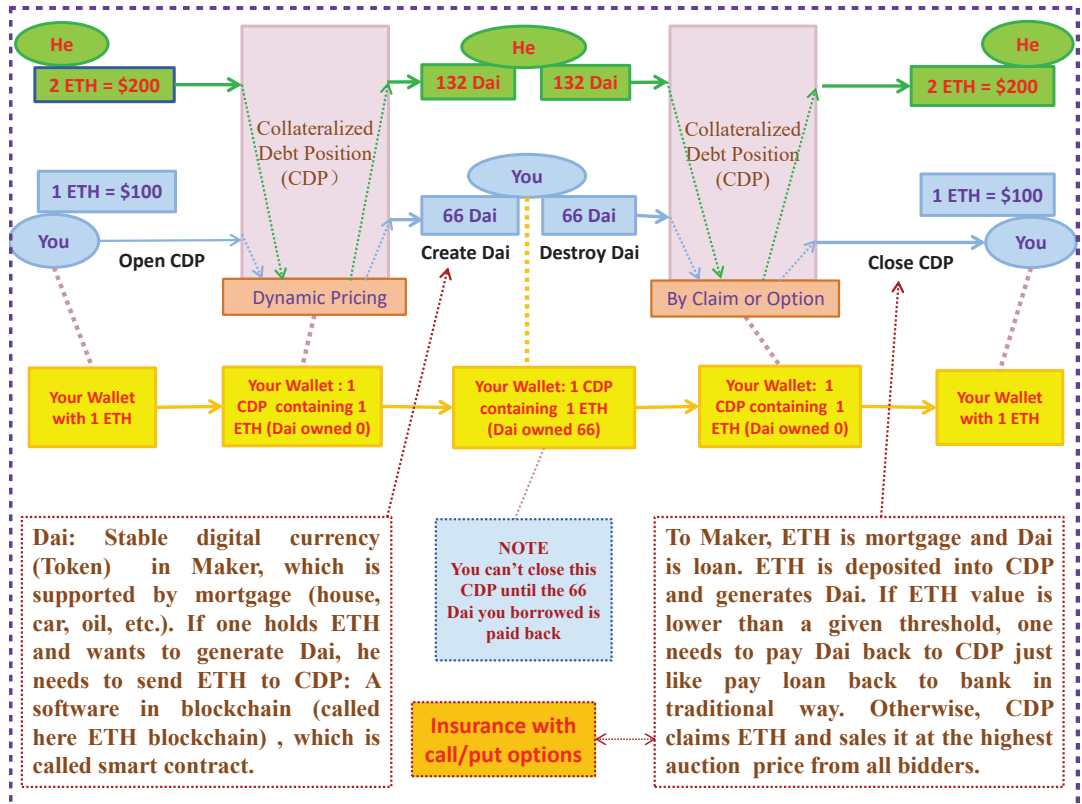}}
\caption{\footnotesize A simple supply chain with finance transactions via DaiCoin (abbreviated as Dai: the first ever designed stable digital currency in the year of 2009). In this figure, the transaction flow charts through Maker for customers He (with green flow chart) and You (with blue flow chart) is presented. The Maker is a blockchain based software management system and ETH means Ethereum.}
\label{daicoin.eps}
\end{figure}
Like a traditional token used in a Casino or a transportation payment system, it is endowed with real value. Furthermore, it can be traced back to the optimal pricing of bits (or ports) in telecommunication managements and admission controls through token buffers in communication networks (see, e.g., Dai~\cite{dai:optrat}, Elwalid and Mitra~\cite{elwmit:anades}). More precisely, around the mid and late 90's, Dai developed the first linked (distributed) data base (currently called blockchain) based optimization algorithms to price and mine bits (or ports) (that can be considered as coins) for 5E (the fifth generation of electronic switch) and many other telecom equipments with capacity constraints and U.S. FCC (Federal Communication Commission) tariff regulations. The purpose of the developments is to help the business leader group in AT$\&$T Bell Labs (now Nokia Bell Labs) to make investment decisions and to aid marketing/sale units to design effective Bell Labs network solutions with the customers' required performance. Along this line, Nakomoto~\cite{nak:bitpee} extended the concept of bit (or port) to the bitcoin in the year of 2008 and Buterin~\cite{but:ethnex} further enhanced this concept to Ethereum in the year of 2013. Note that, for both the bitcoin and the Ethereum, they are still not real stable digital currencies. However, during this evolvement, a claimed Wei Dai (see, e.g., Dai and Nakomoto~\cite{dainak:daisat} and Maker~\cite{mak:looahe}) made his effort to endow the Ethereum with real value and invented DaiCoin (see, e.g., a generalized version of the DaiCoin system in  Figure~\ref{daicoin.eps}, which we designed for the current study. Note that, the claim (or option) function in Figure~\ref{daicoin.eps} is the identical function if the destroying value is 66 for customer You and 132 for customer He). It is worth to point out that the first lawful stable digital currency along this development is the so-called OilCoin (see, e.g., Dai~\cite{dai:plamod}). Since then, this concept and lawful implementations are becoming more and more popular with the emergence of Libra (or digital US dollar), China Central Bank Blockchain Coin (or digital Chinese RMB (Ren Ming Bi)), and European Central Bank Digital Currency as in Figure~\ref{blockchain-fintech} with the targeted supply chain finance service in Figure~\ref{supplychainfinance}.
\begin{figure}[tbh]
%\centerline{\epsfxsize=7.0in\epsfbox{SupplyChainFinance.eps}}
\centerline{\epsfxsize=5.6in\epsfbox{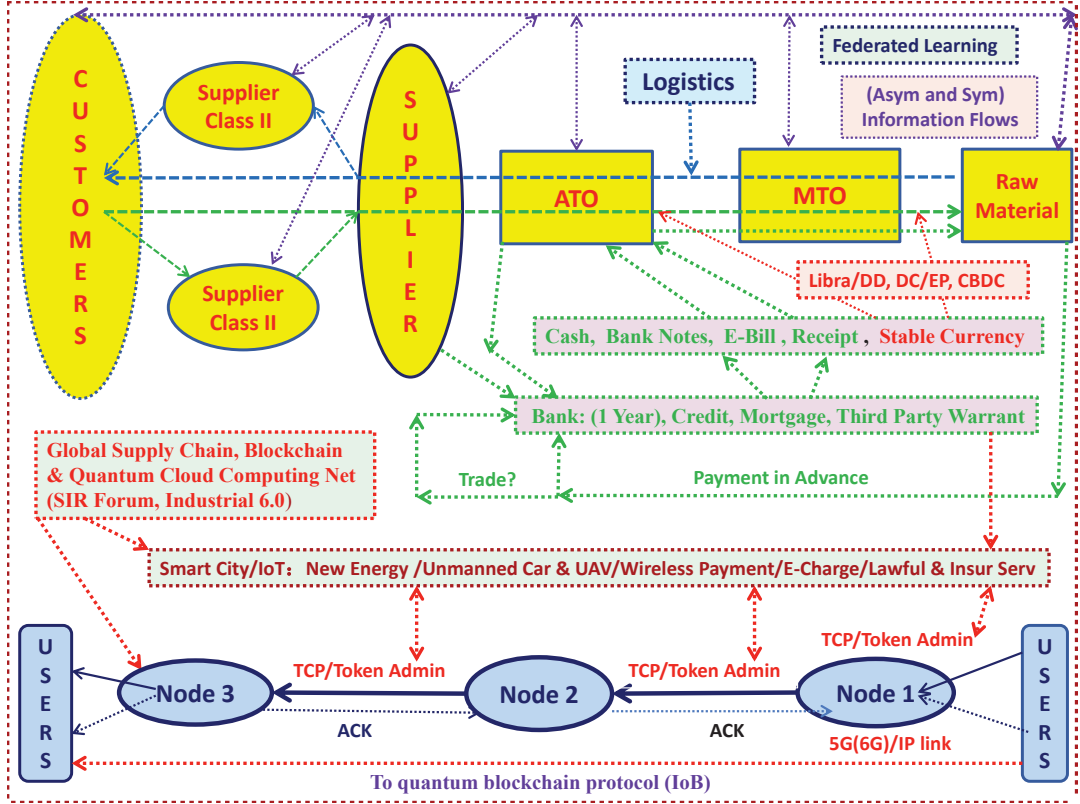}}
\caption{\footnotesize A generalized supply chain with finance transactions via stable digital currencies. In this figure, ATO means assemble to order, MTO means make to order, DD means digital dollar, DC/EP means (China Central Bank) digital currency/electrnic payment, CBDC means (European) Central Bank digital currency, TCP means transmission control protocol, IoB means Internet of Blockchains, Asym and sym mean asymmetry and symmetry respectively.}
\label{supplychainfinance}
\end{figure}

More precisely, this supply chain finance service with additional service lead time involvements can be considered as a generalized online digital payment system extended from the one in Figure~\ref{daicoin.eps}. Note that, a conventional supply chain usually consists of 4 typical service stages to make goods eventually delivered to customers: raw material procurements, make to order (MTO), assemble to order (ATO), and agent sales as designed in the upper-half of Figure~\ref{supplychainfinance} where agents are further classified into two levels of suppliers. Usually, during the procurement and  service stages via our FinTech platform in Figure~\ref{blockchain-fintech}, the cash can not be paid until the delivery of procured products. Thus, the bank notes, e-bill, receipt, etc. through credit, mortgage, and the third party warrant as shown in the middle of Figure~\ref{supplychainfinance} are widely used in real-world practice. To improve the efficiency and security of this type of payments, the stable digital currency such as Libra/DD, DC/EP, and CBDC as designed in Figure~\ref{supplychainfinance} is a suitable choice. Furthermore, as shown in the upper-right corner of Figure~\ref{supplychainfinance}, data information among companies can be asymmetric or symmetric. Frequently, they are not exchangeable. Thus, we will develop a dynamic resource pricing model to solve this problem in the next section, which is a new analytic modeling study to the hot topic on edge-computing and federated learning (see, e.g., Abbas {\it et al.}~\cite{abbzha:mobedg} and Yang {\it et al.}~\cite{yanliu:fedmac}). From the lower-half design in Figure~\ref{supplychainfinance}, we can see that our supply chain system can be mapped into and interact with a information system through wireless 5G/6G network or wireline IP network. Then, many online payments and transactions with lawful services can be handled via our blockchain FinTech system as shown in the middle of Figure~\ref{supplychainfinance}.

\section{System model formulation for pricing stable digital currency}\label{fodpsdc}

In this section, we first introduce the concept of stable digital currency and its evolution. Then, we present the required model primitives for dynamic pricing of stable digital currency together with resource scheduling.

\subsection{Physical system model with quantum primitives}\label{modf}

In this subsection, we present our physical quantum service model with dynamic resource pricing capability and quantum particle arrival flows. Our system model can be used to serve the online payments and transactions with stable digital currencies as proposed in Figures~\ref{daicoin.eps}-\ref{supplychainfinance}. The ``quantum particle flow" is a unified terminology that can be used to model various data information traffics such as the vector-valued data stream with the requirement of voice and image synchronization in a video conference system or the general vector-valued data stream with multiple characteristic indices in statistics, healthcare, economics, quantum computing and quantum robot (see, e.g., Skianis {\it et al.}~\cite{skikon:measta}, Dai~\cite{dai:plamod,dai:quacom},  Toffano and Dubois~\cite{tofdub:quaeig}). More precisely, we will use ``quantum particle flows" to present the Ethereum and cash supply flows, quantum qubit data packet flows, customer and product supply/demand flows appeared in physical banking and insurance services, communication, blockchain and quantum cloud-computing services, supply chain systems, etc. Corresponding to the unified quantum particle flows, the service platform can also be unified as a generalized (IoB based) IoT system or a general FinTech platform model. It owns $V$ number of service pools associated with a set of positive integers ${\cal V}\equiv\{1,...,V\}$ and owns $J$ number of queues for $J$-parallel users corresponding to a set of positive integers ${\cal J}\equiv\{1,...,J\}$). Furthermore, we assume that the buffer storage in each queue is nonnegative. Each pool owns $J_{v}$ number of flexible quantum-computer based parallel-servers with $v$ belonging to a positive integer set ${\cal V}$. Let the prime denote the transpose of a vector or a matrix. Then, associated with the queues, there is a $J$-dimensional quantum particle flow arrival process $A=\{A(t)=(A_{1}(t),...,A_{J}(t))',t\geq 0\}$ and it is called a quantum data packet arrival process. In this situation, $A_{j}(t)$ for each $j\in{\cal J}$, $t\geq 0$, and some positive integer $n\in\{1,2,...\}$ is the number of $n$-qubit data packets that arrive at the $j$th queue during time interval $(0,t]$. In addition, in a real world service system such as a banking service or a supply chain system, the associated input ethereum/cash flows and supply/demand processes can be digitalized and mapped into the $n$-qubit data packet based framework. The size of a quantum data packet is a random number $\zeta\in\{1,2,...\}$. In other words, we can present each quantum data packet by a finite sequence of $n$-qubits $\{|\Phi_{1}\rangle,...,|\Phi_{\zeta}\rangle\}$ where $|\Phi_{i}\rangle$ for each $i\in\{1,2,...,\zeta\}$ is a $n$-qubit. In a real-world service system, this process is referred as a batch arrival process with random batch size $\zeta$.

Our unified blockchain and quantum cloud-computing based service platform system is assumed under an external random environment driven by a stationary FS-CTMC $\alpha=\{\alpha(t),t\in[0,\infty)\}$ with a finite state space ${\cal K}\equiv\{1,...,K\}$. The generator matrix of $\alpha(\cdot)$ is given by $G=(g_{il})$ with $i,l\in{\cal K}$, and
\begin{eqnarray}
&&g_{il}=\left\{\begin{array}{ll}
-\gamma(i)&\mbox{if}\;\;i=l,\\
\gamma(i)q_{il}&\mbox{if}\;\;i\neq l,
\end{array}
\right.
\elabel{generatorm}
\end{eqnarray}
where, $\gamma(i)$ is the holding rate for the continuous time chain staying in a state $i\in{\cal K}$ and $Q=(q_{il})$ is the corresponding transition matrix of its embedded discrete-time Markov chain (see, e.g., Resnick~\cite{res:advsto}). Moreover, define $\tau_{n}$ for each nonnegative integer $n\in\{0,1,...\}$ by
\begin{eqnarray}
&&\tau_{0}\equiv
0,\;\;\tau_{n}\equiv\inf\{t>\tau_{n-1}:\alpha(t)\neq\alpha(t^{-})\}.
\elabel{markovjt}
\end{eqnarray}
To wit, $\tau_{n}$ is a random jump time of the continuous time Markov chain $\alpha(\cdot)$. As in Dai~\cite{dai:plamod}, we model the arrival process $A_{j}(\cdot)$ for each positive integer $j\in{\cal J}$ as a big data flow stream through an TSRRP. For convenience, the definition of an TSRRP is restated as follows.
\begin{definition}\label{dsrp}
A process $A_{j}(\cdot)$ with $j\in{\cal J}\equiv\{1,...,J\}$ is called an TSRRP if $A_{j}(\tau_{n}+\cdot)$ for each $n\in\{0,1,...\}$ is the counting process corresponding to a (conditional) delayed renewal reward process with arrival rate $\lambda_{j}(\alpha(\tau_{n}))$ and mean reward $m_{j}(\alpha(\tau_{n}))$ associated with finite squared coefficients of variations $\alpha^{2}_{j}(\alpha(\tau_{n}))$ and $\zeta^{2}_{j}(\alpha(\tau_{n}))$ during time interval $[\tau_{n},\tau_{n+1})$.
\end{definition}

Now, we let $\{u_{j}(k),k=1,2,...\}$ be the sequence of times between the arrivals of the $(k-1)$th and the $k$th reward batches of packets at the $j$th queue. The associated batch reward is given by $w_{j}(k)$ and all the $n$-qubit data packets arrived with it are indexed in certain successive order. Therefore, we can present the renewal counting process corresponding to the inter-arrival time sequence $\{u_{j}(k),k=1,2,...\}$ for each $j\in{\cal J}$ as follows,
\begin{eqnarray}
&&N_{j}(t)=\sup\left\{n\geq 0:\sum_{k=1}^{n}u_{j}(k)\leq t\right\}.
\elabel{nsjvc}
\end{eqnarray}
Thus, we can restate the definition of an TSRRP $A_{j}(\cdot)$ quantitatively through the expression,
\begin{eqnarray}
&&A_{j}(t)=\sum_{k=1}^{N_{j}(t)}w_{j}(k).
\elabel{ansjvc}
\end{eqnarray}
Each $n$-qubit data packet will first get service in the system and then leave it. The service is managed by a quantum blockchain.
%\begin{figure}[tbh]
%\centerline{\epsfxsize=6.1in\epsfbox{blockchainexample.eps}}
%%\centerline{\epsfxsize=3.5in\epsfbox{blockchainmanager.eps}}
%\caption{\small Quantum-Blockchain Evolution and Processing Chart}
%\label{blockchain}
%\end{figure}
In this blockchain, the service for a $n$-qubit data packet is composed of two parts: security checking and policy computation (or real data
payload transmission). After completing the service, the security information and the policy (or the transmission result) will be stored and copied to all the participating partner nodes for storage and in the meanwhile to produce nonce values and private keys. We call the service associated with the policy computation as a virtue big data service and the service associated with the data payload transmission as a real big data service. Moreover, we denote $\{v_{j}(k),k=1,2,...\}$ to be the sequence of successive arrived packet lengths at queue $j$, which is assumed to be a sequence of strictly positive i.i.d. random variables with average packet length $1/\mu_{j}\in(0,\infty)$ and squared coefficient of variation $\beta_{j}^{2}\in(0,\infty)$. In addition, we suppose that all the inter-arrival and service time processes are mutually (conditionally) independent when the environmental state is fixed. Associated with each $j\in{\cal J}$ and each nonnegative constant $h$, we employ $S_{j}(\cdot)$ to denote the renewal counting process corresponding to $\{v_{j}(k),k=1,2,...\}$. In other words,
\begin{eqnarray}
&&S_{j}(h)=\sup\left\{n\geq 0:\sum_{k=1}^{n}v_{j}(k)\leq h\right\}.
\elabel{sjvc}
\end{eqnarray}
%\subsection{{\textcolor{red}{Performance Measures}}}
%{\textcolor{red}{Every packet (or called job) will get service in
%the platform and then leave the system.}}
Define $Q_{j}(t)$ to be the $j$th queue length with $j\in{\cal J}$ at each time $t\in[0,\infty)$ and $D_{j}(t)$ to be the number of packet departures from the $j$th queue in $(0,t]$.
%{\textcolor{red}{(see, e.g., Figure~\ref{qpools}
%for such an example).}}
%\begin{figure}[tbh]
%\centerline{\epsfxsize=3.5in\epsfbox{smartgridpower.eps}}
%\caption{\small A game platform with parallel-queues and multiple
%cloud-computing service pools}
%\label{qpools}
%\end{figure}
Therefore, the queueing dynamics governing the evolving of the internal qubit data flow in and out within our unified service platform can be modeled by
\begin{eqnarray}
&&Q_{j}(t)=Q_{j}(0)+A_{j}(t)-D_{j}(t),
\elabel{queuelength}
\end{eqnarray}
where, each queue is assumed to have an infinite storage capacity to buffer real or virtue quantum data packets (jobs) arrived from a given user.

Note that, in a DaiCoin and blockchain based mortgage system as shown in Figure~\ref{daicoin.eps}, $Q_{j}(t)$ is the number of Ethereums available at time $t$. In this case, we need to dynamically determine how many Dais should be loaned to customer $j$ for each Ethereum at time $t$ according to the value of $Q(t)$. Similarly, in a banking system as shown in Figure~\ref{blockchain-fintech}, $Q_{j}(t)$ can be the number of loan demands waiting at time $t$. In this case, we need to determine what is the loan interest rate at time $t$ according to the value of $Q(t)$. Furthermore, in communication and quantum cloud-computing based service systems, we need to price the bit service ratio at time $t$ according to the value of $Q(t)$. In all, we need to dynamically price our service in a real-world system according to the evolving of $Q(t)$ with the evolution of time $t$. For convenience, we will use the unified terminology ``price $P_{j}(t)$" to denote the price (the number of Dais or interest ratio) associated with $Q_{j}(t)$ at time $t$. In economics, there are different pricing functions with respect to $Q(t)$ (see, e.g., Dai and Jiang~\cite{daijia:stoopt}). Here, we assume that $P_{j}(t)$ is a positive function in terms of $Q_{j}(t)$ and $\alpha(t)$, i.e.,
\begin{eqnarray}
&&P_{j}(t)=f_{j}(Q_{j}(t),\alpha(t)).
\elabel{pricefun}
\end{eqnarray}
In addition, we suppose that $f_{j}(\cdot,\cdot)$ in \eq{pricefun} is Lipschitz continuous with respect to $Q_{j}(t)$. Then, we can introduce a utility (or a hash) function with respect to the valued queue length $P_{j}(t)Q_{j}(t)$ for user $j\in{\cal J}$ at each service pool $v\in{\cal V}$ as follows,
\begin{eqnarray}
&&U_{vj}(P(t)Q(t),\Lambda(t))\;\;\mbox{with}\;\;P(t)Q(t)=(P_{1}(t)Q_{1}(t),...,P_{J}(t)Q_{J}(t)),
\elabel{utilityvj}
\end{eqnarray}
where, $P(t)=(P_{1}(t),...,P_{J}(t))$ and $\Lambda(t)=(\Lambda_{1}(t),...\Lambda_{J}(t))$. Moreover, $\Lambda_{j}(t)$ for each $t\in[0,\infty]$ and $j\in{\cal J}$ is the summation of all service rates allocated to the $j$th user at time $t$ from all possible pools and servers.

Now, we define $W(t)$ and $W_{j}(t)$ to be the (expected) total workload in the system at time $t$ and the one associated with user $j$ at time $t$, to wit,
\begin{eqnarray}
&&W(t)=\sum_{j=1}^{J}W_{j}(t),\;\;\;\;\;\;\;W_{j}(t)=\frac{Q_{j}(t)}{\mu_{j}}.
\elabel{wyte}
\end{eqnarray}
In the following study, we will use $W(t)$ and $Q(t)$ as performance measures, $f=(f_{1},...,f_{J})$ in \eq{pricefun} as pricing function, and $\{U_{vj},j\in{\cal J},v\in{\cal V}\}$ in \eq{utilityvj} as utility (or hash) functions to propose a joint dynamical pricing and rate scheduling policy $(P,\Lambda)$ with users' selection at each time point for different service pools and servers to all the users in order that the total workload $W(t)$ and its corresponding total cost are minimized. Here we note that the available resources in our current system are generally transformed into service rates although they can be interpreted as other forms, e.g., power in an MIMO wireless channel or in a quantum-computing $\&$ measurement channel. Furthermore, we assume that the available resources from different pools and servers can be flexibly allocated and shared between the system and users, i.e., the system operates under a concurrent resource occupancy service regime. Based on these facts, we can define $T_{j}(t)$ to be the cumulative amount of service given to the $j$th queue up to time $t$, i.e.,
\begin{eqnarray}
&&T_{j}(t)=\int_{0}^{t}\Lambda_{j}(P(s)Q(s),\alpha(s))ds,
\elabel{tjqalpha}
\end{eqnarray}
Here, we remark that, $\Lambda_{j}$ is given in a feedback control form and it depends on the current price $P(s)$, the current queue length $Q(s)$, and the system state $\alpha(s)$ at a given time $s$. Hence, if we let $S_{j}(t)$ be the total number of jobs (packets) that finishes service in the system by time $t$, we know that $D_{j}(t)=S_{j}(T_{j}(t))$.

\section{RDRS model for dynamic resource pricing and scheduling}\label{pmodeling}

TSRRPs in Definition~\ref{dsrp} can effectively model big data arrival streams. However, it is difficult to directly conduct the analysis of the associated physical queueing model in \eq{queuelength} or its related physical workload model in \eq{wyte} due to the non-Markovian characteristics of TSRRPs. Thus, in this paper, we will develop a scheme to establish the RDRS model corresponding to our newly designed game-competition based dynamic resource pricing and scheduling policy by considering our queueing system under the asymptotic regime, where it is heavily loaded (load balanced), i.e., under the so-called heavy traffic condition. Furthermore, we will prove the correctness of RDRS modeling via diffusion approximation while we will also show the effectiveness of the identified model for our newly proposed pricing and scheduling policy by presenting simulation case studies. The corresponding simulation results are displayed in Figures~\ref{3simexamplei}-\ref{3simexampleii} and Figure~\ref{simexamplei} and their interpretations are presented in Subsection~\ref{singlepoolexample}.
\begin{figure}[tbh]
%\centerline{\epsfxsize=9.0in\epsfbox{s-i15-2-07-Gp-5k.eps}}
\centerline{\epsfxsize=7.1in\epsfbox{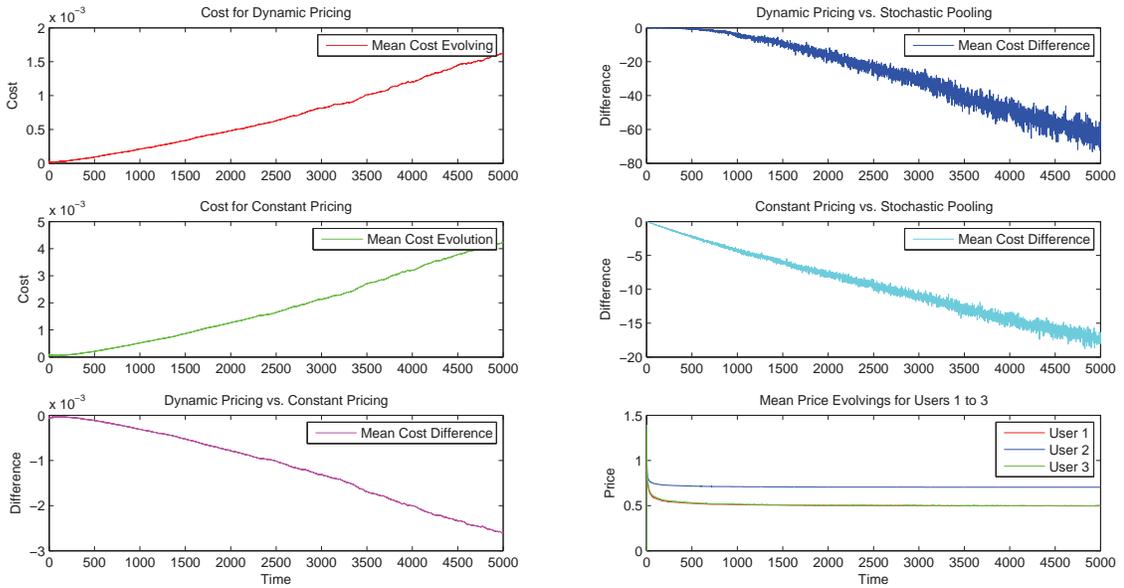}}
\caption{\footnotesize In this simulation, the number of simulation iterative times is $N=6000$, the simulation time interval is $[0,T]$ with $T=20$, which is further divided into $n=5000$ subintervals as explained in Subsection~\ref{singlepoolexample}. Other values of simulation parameters introduced in Definition~\ref{rdrsd} and Subsubsection~\ref{illex} are as follows: initialprice1=2.25, initialprice2=1.5, initialprice3=2.25, upperboundprice1=4, upperboundprice2=2, upperboundprice3=4, lowerboundprice1=0.49, lowerboundprice2=0.7, lowerboundprice3=0.49, queuepolicylowerbound1=0, queuepolicylowerbound2=0, queuepolicylowerbound3=0, $\lambda_{1}=10/3$, $\lambda_{2}=5$, $\lambda_{3}=10/3$, $m_{1}=3$, $m_{2}=1$, $m_{3}=3$, $\mu_{1}=1/10$, $\mu_{2}=1/20$, $\mu_{3}=1/10$, $\alpha_{1}=\sqrt{10}$, $\alpha_{2}=\sqrt{20}$, $\alpha_{3}=\sqrt{10}$, $\beta_{1}=\sqrt{10}$, $\beta_{2}=\sqrt{20}$, $\beta_{3}=\sqrt{10}$, $\zeta_{1}=1$, $\zeta_{2}=\sqrt{2}$, $\zeta_{3}=1$, $\rho_{1}=\rho_{2}=\rho_{3}=1000$, $\theta_{1}=-1$, $\theta_{2}=-1.2$, $\theta_{3}=-1$.}
\label{3simexamplei}
\end{figure}
\begin{figure}[tbh]
%\centerline{\epsfxsize=9.0in\epsfbox{s-N6000-I111-U111-L111-Gp.eps}}
\centerline{\epsfxsize=7.1in\epsfbox{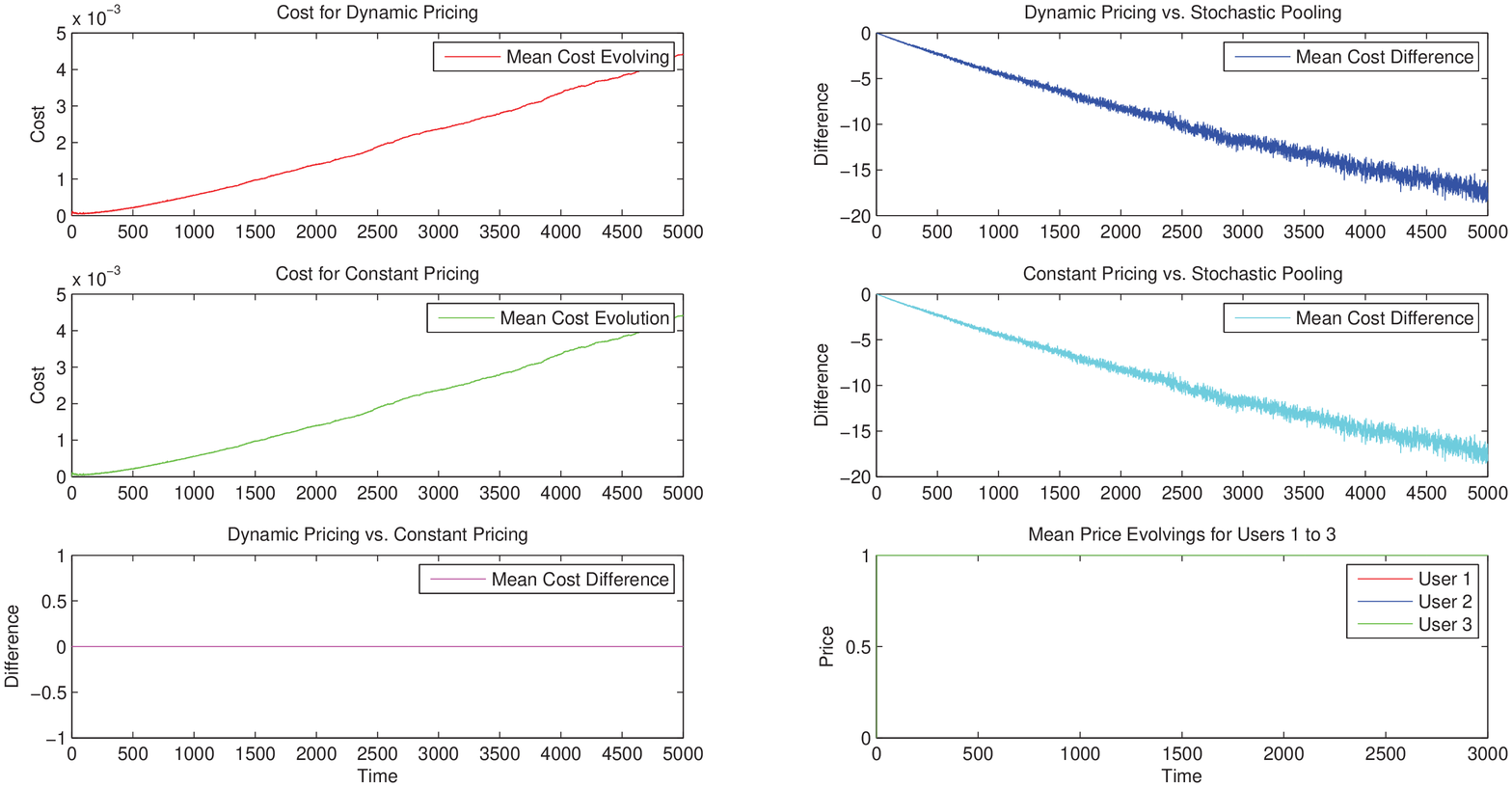}}
\caption{\footnotesize In this simulation, the number of simulation iterative times is $N=6000$, the simulation time interval is $[0,T]$ with $T=20$, which is further divided into $n=5000$ subintervals as explained in Subsection~\ref{singlepoolexample}. Other values of simulation parameters introduced in Definition~\ref{rdrsd} and Subsubsection~\ref{illex} are as follows: initialprice1=1, initialprice2=1, initialprice3=1, upperboundprice1=1, upperboundprice2=1, upperboundprice3=1, lowerboundprice1=1, lowerboundprice2=1, lowerboundprice3=1, queuepolicylowerbound1=0, queuepolicylowerbound2=0, queuepolicylowerbound3=0, $\lambda_{1}=10/3$, $\lambda_{2}=5$, $\lambda_{3}=10/3$, $m_{1}=3$, $m_{2}=1$, $m_{3}=3$, $\mu_{1}=1/10$, $\mu_{2}=1/20$, $\mu_{3}=1/10$, $\alpha_{1}=\sqrt{10}$, $\alpha_{2}=\sqrt{20}$, $\alpha_{3}=\sqrt{10}$, $\beta_{1}=\sqrt{10}$, $\beta_{2}=\sqrt{20}$, $\beta_{3}=\sqrt{10}$, $\zeta_{1}=1$, $\zeta_{2}=\sqrt{2}$, $\zeta_{3}=1$, $\rho_{1}=\rho_{2}=\rho_{3}=1000$, $\theta_{1}=-1$, $\theta_{2}=-1.2$, $\theta_{3}=-1$.}
\label{3simexampleii}
\end{figure}

%\subsection{Modeling Performance Measures as RDRSs}
\subsection{Main theorem via game-competition based smart contract}

In this subsection, we first present our main claim in terms of our RDRS modeling under a smart contract policy. Second, for convenience, we introduce the definition of RDRS model. More precisely, for each $t\geq 0$ and $j\in{\cal J}$, we introduce two sequences of diffusion-scaled processes $\hat{Q}^{r}(\cdot)$ and $\hat{W}^{r}(\cdot)$ by
%and $\hat{W}^{r}_{j}(\cdot)$ by
\begin{eqnarray}
&&\hat{Q}_{j}^{r}(t)\equiv\frac{Q_{j}^{r}(r^{2}t)}{r},\;\;\;\;\;\;\;\hat{W}^{r}(t)\equiv\frac{W^{r}(r^{2}t)}{r},
%\;\;\;\hat{W}^{r}_{j}(t)\equiv\frac{W^{r}_{j}(r^{2}t)}{r}
\elabel{rsqueue}
\end{eqnarray}
where, $\{r,r\in{\cal R}\}$ is supposed to be a strictly increasing sequence of positive real numbers and tends to infinity. Then, our main claim can be presented as follows.

%\begin{claim}\label{claimstate}
The sequence of $2$-tuple scaled processes in \eq{rsqueue} corresponding to a game-competition based dynamic resource pricing and scheduling policy with users' selection, which is designed in the subsequent subsection, converges jointly in distribution. More precisely, under the heavy traffic condition described in Section~\ref{rdrsm}, we have that
\begin{eqnarray}
&&(\hat{Q}^{r}(\cdot),\hat{W}^{r}(\cdot))
\Rightarrow(\hat{Q}(\cdot),\hat{W}(\cdot))
\;\;\;\mbox{along}\;\;\;r\in{\cal R},
\elabel{aI-qwyweakc}
\end{eqnarray}
where, $\hat{W}(\cdot)$ is presented by an RDRS model and $\hat{Q}(\cdot)$ is an asymptotic queue policy process with dynamic pricing globally over $[0,\infty)$ through a saddle point to zero-sum game-competition problem and a Pareto minimal-dual-cost Nash equilibrium point to a non-zero-sum game-competition problem.
%\end{claim}
\begin{definition}\label{rdrsd}
A $u$-dimensional stochastic process $\hat{Z}(\cdot)$ with $u\in{\cal J}$ is claimed as an RDRS with oblique reflection if it can be uniquely represented as
\begin{eqnarray}
&&\left\{\begin{array}{ll}
\hat{Z}(t)&=\;\;\;\hat{X}(t) +\int_{0}^{t}R(\alpha(s),s)d\hat{Y}(s)\geq 0,\\
%\elabel{mainmodelz}
d\hat{X}(t)&=\;\;\;b(\alpha(t),t)dt+\sigma^{E}(t)d\hat{H}^{E}(t)+\sigma^{S}(t)d\hat{H}^{S}(t).
%\elabel{hatzx}
\end{array}
\right.
\elabel{mainmodelz}
\end{eqnarray}
Furthermore, $b(\alpha(t),t)=(b_{1}(\alpha(t),t),...,b_{u}(\alpha(t),t)'$ is a $u$-dimensional vector, $\sigma^{E}(t)$ and $\sigma^{S}(t)$ are
$u\times J$ matrices, $R(\alpha(t),t)$ with $t\in R_{+}$ is a $u\times u$ matrix, and $(\hat{Z}(\cdot),\hat{Y}(\cdot))$ is a coupled a.s. continuous solution of \eq{mainmodelz} with the following properties for each $j\in\{1,...,u\}$,
\begin{eqnarray}
&\left\{\begin{array}{ll}
\hat{Y}_{j}(0)=0;\\
\mbox{Each component}\;\;\hat{Y}_{j}(\cdot)\;\;\mbox{of}\;\;
\hat{Y}(\cdot)=(\hat{Y}_{1}(\cdot),...,\hat{Y}_{u}(\cdot))'\;\;\mbox{is non-decreasing};\\
\mbox{Each component}\;\;\hat{Y}_{j}(\cdot)\;\;\mbox{can increase only at a time}\;\;t\in[0,\infty)\;\;\mbox{that}\;\;\hat{Z}_{j}(t)=0,\;\mbox{i.e.},\\
\;\;\;\;\;\;\;\;\;\;\;\;\;\;\;\;\;\;\;\;\;\;\;\;\;\;\;\;\;\;\;\;\;\;\;\;\;\;\;\;\;\;\;\;\;\;\;\;
\int_{0}^{\infty}\hat{Z}_{j}(t)d\hat{Y}_{j}(t)=0.
\end{array}
\right.
\nonumber
\end{eqnarray}
In addition, a solution to the RDRS in \eq{mainmodelz} is called a strong solution if it is in the pathwise sense and is called a weak solution if it is in the sense of distribution.
\end{definition}

In terms of the well-posedness of an RDRS, readers are referred to a general discussion in Dai~\cite{dai:unisys}. Furthermore, in Definition~\ref{rdrsd}, the stochastic processes $B^{E}(\cdot)$ and $B^{S}(\cdot)$ are respectively two $J$-dimensional standard Brownian motions, which are independent each other. For each state $i\in{\cal K}$ and a time $t\in[0,\infty)$, the nominal arrival rate vector $\lambda(i)$, the mean
reward vector $m(i)$ , the nominal throughput vector $\rho(i)$, and a constant parameter vector $\theta(i)$ are given as follows,
\begin{eqnarray}
&&\left\{\begin{array}{ll}
\lambda(i)&=\;\;\;(\lambda_{1}(i),...,\lambda_{J}(i))',\\
%\elabel{lambdae}\\
m(i)&=\;\;\;(m_{1}(i),...,m_{J}(i))',\\
%\elabel{mlambdae}\\
\rho(i)&=\;\;\;\left(\rho_{1}(i),...,\rho_{J}(i)\right)',\\
%\elabel{barvc}\\
\theta(i)&=\;\;\;(\theta_{1}(i),...,\theta_{J}(i))'.
%\elabel{vectorpar}
\end{array}
\right.
\elabel{lambdae}
\end{eqnarray}
The covariance matrices are given by
\begin{eqnarray}
&&\left\{\begin{array}{ll}
\Gamma^{E}(i)&=\;\;\;\left(\Gamma^{E}_{kl}(i)\right)_{J\times J}\\
%\elabel{covmatrixo}\\
&\equiv\;\;\;\mbox{diag}\left(\lambda_{1}(i)m_{1}^{2}(i)\zeta^{2}_{1}(i)
+\lambda_{1}(i)m_{1}(i)\alpha_{1}^{2},\right.\\
%\nonumber\\
&\;\;\;\left.\;\;\;\;...,\lambda_{J}(i)m_{J}^{2}(i)
\zeta^{2}_{J}(i)+\lambda_{J}(i)m_{J}(i)\alpha^{2}_{J}\right),\\
%\nonumber\\
\Gamma^{S}(i)&=\;\;\;\left(\Gamma^{S}_{kl}(i)\right)_{J\times J}\\
%\elabel{covmatrix}\\
&\equiv\;\;\;\mbox{diag}\left(\lambda_{1}(i)m_{1}(i)\beta_{1}^{2},...,
\lambda_{J}(i)m_{J}(i)\beta_{J}^{2}\right).
%\nonumber
\end{array}
\right.
\elabel{covmatrixo}
\end{eqnarray}
The It$\hat{o}$'s integrals with respect to the Brownian motions are defined as
\begin{eqnarray}
&&\left\{\begin{array}{ll}
\hat{H}^{e}(t)&=\;\;\;\left(\hat{H}^{e}_{1}(t)',...,\hat{H}^{e}_{J}(t)\right)'\;\;\mbox{with}\;\;e\in\{E,S\},\\
%\elabel{bbvector}\\
\hat{H}^{e}_{j}(t)&=\;\;\;\int_{0}^{t}\sqrt{\Gamma^{e}_{jj}(\alpha(s))}dB^{e}_{j}(s).
%\elabel{bbvectorI}
\end{array}
\right.
\elabel{bbvector}
\end{eqnarray}

\subsection{A 3-Stage dynamic pricing and scheduling policy}

In this subsection, we design a 3-stage users' selection, dynamic pricing, and rate scheduling policy through mixed zero-sum and non-zero-sum game-competitions myopically at each time point for the purpose as stated in the previous subsection. To be more illustrative, we begin with two policy examples.

\subsubsection{Two illustrative policy examples}\label{illex}

In this subsubsection, we present a 2-stage dynamic pricing and rate scheduling example and a 3-stage users' selection, dynamic pricing, and rate scheduling example based on a single-pool service system as shown in Figure~\ref{numexample}.
\begin{figure}[tbh]
%\centerline{\epsfxsize=7.4in\epsfbox{channelpricecap3U.eps}}
\centerline{\epsfxsize=5.6in\epsfbox{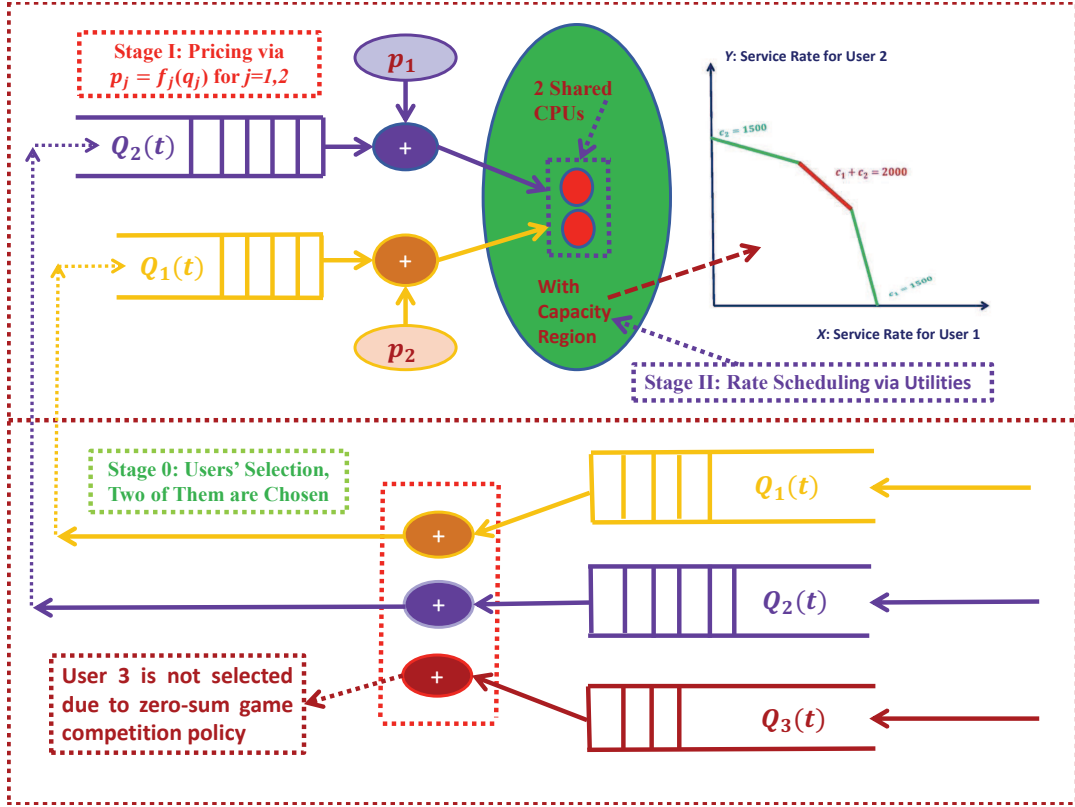}}
\caption{\footnotesize 2-stage dynamic pricing and rate scheduling for a single pool service system with 3-users}
\label{numexample}
%\centerline{\epsfxsize=2.0in\epsfbox{paretooptimal.eps}}
%\caption{\small An example of Pareto optimal Nash
%equilibrium policy}
\end{figure}
Hence, we will omit the associated pool index $v$.

The first example is corresponding to the upper graph in Figure~\ref{numexample}, which can be used to model the DaiCoin based digital payment system with two types of Ethereums (corresponding to two DaiCoins) as in Figure~\ref{daicoin.eps}. It can also be used to model an MIMO channel shared by two-users or a quantum computer with two eigenmodes. In this case, we are interested in the problem about how to price the two users and conduct the computing rate (i.e., power) resource allocation cooperatively inside a service system. More precisely, we take $V=1$ with $J=2$ and assume that the state space of the FS-CTMC $\alpha(t)$ defined in Subsection~\ref{modf} consists only of a single state (i.e., $\alpha(t)\equiv 1$ for all $t\in[0,\infty)$). In an MIMO wireless environment, this case is associated with the so-called pseudo static channels. Then, the capacity region denoted by ${\cal R}$ is supposed to be a non-degenerate convex one confined by five boundary lines including the two ones on $x$-axis and $y$-axis as shown in the upper-right graph of Figure~\ref{numexample}. The capacity upper bound of the region satisfies $c_{1}+c_{2}=2000$. This region is corresponding to a degenerate fixed MIMO wireless channel of the generally randomized one in Dai~\cite{dai:optrat}. For each price vector $p=(p_{1},p_{2})\in R_{+}^{2}=[0,\infty)\times[0,\infty)$ corresponding to the process $P(t)$ in \eq{pricefun} and the queue length vector $q=(q_{1},q_{2})\in R_{+}^{2}=[0,\infty)\times[0,\infty)$ corresponding to the process $Q(t)$ defined in \eq{queuelength} at a particular time point,
%we take the pricing functions in \eq{pricefun} for user 1 and user 2 to be the well-known power functions given respectively by
%\begin{eqnarray}
%&&p_{1}=f_{1}(q_{1})=\frac{1000}{q^{2}_{1}},\;\;\;\;\;\;p_{2}=f_{2}(q_{2})=\sqrt{\frac{6000}{q^{3}_{2}}}.
%\elabel{exactpricefun}
%\end{eqnarray}
we take the utility functions in terms of rate vector $c=(c_{1},c_{2})\in{\cal R}$ for user 1 and user 2 respectively by
\begin{eqnarray}
&&U_{1}(pq,c)=U_{1}(p_{1}q_{1},c_{1})=p_{1}q_{1}\ln(c_{1}),\;\;\;U_{2}(pq,c)=U_{2}(p_{2}q_{2},c_{2})=-\frac{(p_{2}q_{2})^{2}}{c^{2}_{2}},
\elabel{numutility}
\end{eqnarray}
where, $ln(\cdot)$ is the logarithm function with the base $e$. Note that, the utility functions $U_{1}$ and $U_{2}$ in \eq{numutility} are called proportionally fair and minimal potential delay allocations respectively, which are widely used in communication systems (see, e.g., Dai~\cite{dai:optrat}, Ye and Yao~\cite{yeyao:heatra}). Furthermore, in a (quantum) blockchain system, these utility functions can be considered as generalized hash functions to replace the currently used random number generators to generate partial nonce values and private keys, i.e., $((p_{1},p_{2}),(q_{1},q_{2}),(c_{1},c_{2}))$.

\begin{example}\label{example1}
For the upper graph case in Figure~\ref{numexample} and by the utility functions in \eq{numutility}, we can propose a 2-stage pricing and rate-scheduling policy at each time point $t\in[0,\infty)$ by a Pareto maximal-utility Nash equilibrium point to the following non-zero-sum game problem
\begin{eqnarray}
&&\max_{c\in{\cal R}}U_{j}(pq,c)\;\;\mbox{for each}\;\;j\in\{0,1,2\}\;\;\mbox{and a fixed}\;\;pq\in R_{+}^{2},
\elabel{adafn}
\end{eqnarray}
where, $U_{0}(pq,c)=U_{1}(pq,c)+U_{2}(pq,c)$. To wit, if $c^{*}=(c_{1}^{*},c_{2}^{*})$ is a solution to the game problem in \eq{adafn}, we have that
\begin{eqnarray}
&&\left\{\begin{array}{ll}
U_{0}(pq,c^{*})\geq U_{0}(pq,c),&\\
%\elabel{upolicyIn}\\
U_{1}(pq,c^{*})\geq U_{1}(pq,c^{*}_{-1})&\mbox{with}\;\;\;c^{*}_{-1}=(c_{1},c^{*}_{2}),\\
%\elabel{upolicyIIn}\\
U_{2}(pq,c^{*})\geq U_{2}(pq,c^{*}_{-2})&\mbox{with}\;\;\;c^{*}_{-2}=(c^{*}_{1},c_{2}).
%\elabel{upolicyIIIn}
\end{array}
\right.
\elabel{upolicyIn}
\end{eqnarray}
Furthermore, it follows from the inequalities in \eq{upolicyIn} that, if a game player's (i.e., a user's) rate service policy is unilaterally changed, his utility cannot be improved.
\end{example}

The second example is by adding Stage 0 for users' selection in the lower graph of Figure~\ref{numexample}. Comparing with the first case with $J=2$, we here consider a 3-user case (i.e., $J=3$) and add one more user selection layer. At each time point, we choose two of the three users for service according to a zero-sum game competition policy. When any two users $i,j\in\{1,2,3\}$ with $i\neq j$ are selected, they will be served based on a non-zero-sum game competition policy. The capacity upper bound of the corresponding capacity region satisfies $c_{i}+c_{j}=2000$ as in the first 2-user case. Furthermore, suppose that, at a particular time point, there is a price vector $p=(p_{1},p_{2},p_{3})\in R_{+}^{3}$ corresponding to the process $P(t)$ in \eq{pricefun} and a queue length vector $q=(q_{1},q_{2},q_{3})\in R_{+}^{3}$ corresponding to the process $Q(t)$ defined in \eq{queuelength}. Then, for each $(c_{i},c_{j})\in{\cal R}$, the corresponding utility functions are taken as in \eq{numutility} if $i,j\in\{1,2\}$. However, if $i=3$ or $j=3$, the corresponding utility function is taken to be the following one,
\begin{eqnarray}
&&U_{3}(pq,c)=U_{3}(p_{3}q_{3},c_{3})=p_{3}q_{3}\ln(c_{3}).
\elabel{anumutility}
\end{eqnarray}

\begin{example}\label{example2}
For the second case corresponding to both the upper and lower graphs in Figure~\ref{numexample} and by the utility functions in \eq{numutility} and \eq{anumutility}, we can design a 3-stage user-selection, pricing and rate-scheduling policy myopically at each time point $t\in[0,\infty)$, which involves two steps as follows. First, we choose two users for service by a saddle point policy via the solution to the zero-sum game problem,
\begin{eqnarray}
&&\;\;\;\;\max_{c\in{\cal R}}U_{0}(pq,c),\;\max_{c\in{\cal R}}U_{0j}(pq,c),\;\max_{c\in{\cal R}}\left(-U_{0j_{1}}(pq,c)\right),\;\max_{c\in{\cal R}}\left(-U_{0j_{2}}(pq,c)\right)
\elabel{adaf}
\end{eqnarray}
for each $j\in\{1,2,3\}$, $j_{1}\in\{1,2,3\}\setminus\{j\}$, $j_{2}\in\{1,2,3\}\setminus\{j,j_{1}\}$, and a fixed $pq\in R_{+}^{3}$, where,
\begin{eqnarray}
&&\left\{\begin{array}{ll}
U_{0}(pq,c)&=\;\;U_{1}(pq,c)+U_{2}(pq,c)+U_{3}(pq,c),\\
%\nonumber\\
U_{01}(pq,c)&=\;\;U_{1}(pq,c)+U_{2}(pq,c),\\
%\elabel{twoI}\\
U_{02}(pq,c)&=\;\;U_{1}(pq,c)+U_{3}(pq,c),\\
%\elabel{twoII}\\
U_{03}(pq,c)&=\;\;U_{2}(pq,c)+U_{3}(pq,c).
%\elabel{twoIII}
\end{array}
\right.
\elabel{twoI}
\end{eqnarray}
In other words, if $c^{*}=(c_{1}^{*},c_{2}^{*},c_{3}^{*})$ is a solution to the game problem in \eq{adaf}, and if
\begin{eqnarray}
&&c^{*}_{-j}=\left\{\begin{array}{ll}
(c_{j},c^{*}_{j_{1}},c^{*}_{j_{2}})&\mbox{if}\;\;j=1,\\
(c^{*}_{j_{1}},c_{j},c^{*}_{j_{2}})&\mbox{if}\;\;j=2,\\
(c^{*}_{j_{1}},c^{*}_{j_{2}},c_{j})&\mbox{if}\;\;j=3,
\end{array}
\right.
\nonumber
\end{eqnarray}
then, for a fixed $pq\in R_{+}^{3}$, we have that
\begin{eqnarray}
&&\left\{\begin{array}{ll}
U_{0}(pq,c^{*})&\geq\;\;U_{0}(pq,c),\\
%\elabel{upolicyI}\\
U_{0j}(pq,c^{*})&\geq\;\;U_{0j}(pq,c^{*}_{-j}),\\
%\elabel{upolicyII}\\
-U_{0j_{1}}(pq,c^{*})&\geq\;\;-U_{0j_{1}}(pq,c^{*}_{-j_{1}}),\\
%\elabel{upolicyIII}\\
-U_{0j_{2}}(pq,c^{*})&\geq\;\;-U_{0j_{2}}(pq,c^{*}_{-j_{2}}).
%\elabel{upolicyIV}
\end{array}
\right.
\elabel{upolicyI}
\end{eqnarray}
Second, when two users corresponding to the summation $U_{0j}=U_{k}+U_{l}$ for an index $j\in\{1,2,3\}$ with two associated indices $k,l\in\{1,2,3\}$ as in one of \eq{twoI} are selected, we can propose a 2-stage pricing and rate-scheduling policy at each time point by a Pareto maximal-utility Nash equilibrium point to the non-zero-sum game problem for a fixed $pq\in R_{+}^{3}$,
\begin{eqnarray}
&&\max_{c\in{\cal R}}U_{0j}(pq,c),\;\;\max_{c\in{\cal R}}U_{k}(pq,c),\;\;\max_{c\in{\cal R}}U_{l}(pq,c).
\elabel{aadafn}
\end{eqnarray}
To wit, if $c^{*}=(c_{k}^{*},c_{l}^{*})$ is a solution to the game problem in \eq{aadafn}, we have that
\begin{eqnarray}
&&\left\{\begin{array}{ll}
U_{0j}(pq,c^{*})\geq U_{0j}(pq,c),&\\
%\elabel{aupolicyIn}\\
U_{k}(pq,c^{*})\geq U_{1}(pq,c^{*}_{-k})&\mbox{with}\;\;\;c^{*}_{-k}=(c_{k},c^{*}_{l}),\\
%\elabel{aupolicyIIn}\\
U_{l}(pq,c^{*})\geq U_{l}(pq,c^{*}_{-l})&\mbox{with}\;\;\;c^{*}_{-l}=(c^{*}_{k},c_{l}).
%\elabel{aupolicyIIIn}
\end{array}
\right.
\elabel{aupolicyIn}
\end{eqnarray}
\end{example}

\subsubsection{General service capacity region}

In our quantum cloud-computing based IoB network system, the jobs in the $j$th queue for each $j\in{\cal J}$ may be served at the same time by a random but at most $V_{j}$ ($\leq V$) number of service pools corresponding to selected utility (hash) functions at a particular time point. With this simultaneous service mechanism, the total service rate for the $j$th queue at the time point is the summation of the rates from all the pools possibly to serve the $j$th queue. More precisely, we index these pools by a subset ${\cal V}(j)$ of the set ${\cal V}$ as follows,
\begin{eqnarray}
&&{\cal V}(j)\equiv\Big\{v_{1j},...,v_{V_{j}j}\Big\}\subseteq{\cal V},
\elabel{vindex}
\end{eqnarray}
where, $v_{lj}$ with $l\in\{1,...,V_{j}\}$ denotes the $v_{lj}$th pool in ${\cal V}(j)$. In the same way, a pool denoted by $v\in{\cal V}$ can possibly serve at most $J_{v}$ number of job classes represented by a subset ${\cal J}(v)$ of the set ${\cal J}$, i.e.,
\begin{eqnarray}
&&{\cal J}(v)\equiv\Big\{j_{v1},...,j_{vJ_{v}}\Big\}\subseteq{\cal J},
\elabel{jindex}
\end{eqnarray}
where, $j_{vl}$ with $l\in\{1,...,J_{v}\}$ indexes the $j_{vl}$th job class in ${\cal J}(v)$. In the pool $v$, there are $J_{v}$ number of flexible parallel-servers with rate allocation vector
\begin{eqnarray}
&&c_{v\cdot}(t)=(c_{j_{v1}}(t),...,c_{j_{vJ_{v}}}(t))',
\elabel{ratevector}
\end{eqnarray}
where, $c_{j_{vl}}(t)$ with $l\in\{1,...,J_{v}\}$ is the assigned service rate to the $j_{vl}$th user at pool $v$ and time $t$. Similarly, corresponding to the $l\in\{1,...,J_{v}\}$, we will also denote the rate $c_{j_{vl}}(t)$ by $c_{vj}(t)$ for an index $j\in{\cal J}(v)$.

Note that, the vector in \eq{ratevector} takes values in a capacity region ${\cal R}_{v}(\alpha(t))$ driven by the FS-CTMC $\alpha=\{\alpha(t),t\in[0,\infty)\}$.
%\begin{figure}[tbh]
%\centerline{\epsfxsize=3.0in\epsfbox{threeregion.eps}}
%\caption{\small A 3-user capacity set in $R_{+}^{3}$ for
%a cooperative MIMO wireless channel}
%\label{threeuserregion}
%\end{figure}
%%removed on 02/10/2020
%\begin{figure}[tbh]
%\centerline{\epsfxsize=7.0in\epsfbox{cpu-capacity.eps}}
%\caption{\footnotesize A 3-user capacity set in $R_{+}^{3}$ for a cooperative
%MIMO wireless channel in the left graph; A degenerate 3-user capacity set
%in $R_{+}^{3}$ for a cloud-processors-sharing system in the right graph;
%The upper graph is adapted from Figure~\ref{quantumchannel}.
%%{\textcolor{red}{In each of these two graphs, $x$ along $x$-axis denotes
%%the service rate for user 1, $y$ along $y$-axis denotes the service rate
%%for user 2, and $z$ along $z$-axis denotes the service rate for user 3.}}
%}
%\label{threeclouduserregion}
%\end{figure}.
For each given $i\in{\cal K}$ and $v\in{\cal V}$, the set ${\cal R}_{v}(i)$ is a convex region containing the origin and has $L_{v}$ $(>J_{v})$ boundary pieces (see, e.g.,  the upper-left graph in Figure~\ref{numexample}). In this region, every point is defined according to the associated users, i.e., $x=(x_{j_{v1}},...,x_{j_{vJ_{v}}})$. On the boundary of ${\cal R}_{v}(i)$ for each $i\in{\cal K}$, $J_{v}$ of them are $(J_{v}-1)$-dimensional linear facets along the coordinate axes. The other ones denoted by ${\cal O}_{v}(i)$ are located in the interior of $R^{J_{v}}_{+}$. It is called the {\it capacity surface} of ${\cal R}_{v}(i)$ and it has $B_{v}=L_{v}-J_{v}\;(>0)$ linear or smooth curved facets $h_{vk}(c_{v\cdot},i)$ on $R_{+}^{J_{v}}$ for $k\in{\cal U}_{v}\equiv\{1,2,...,B_{v}\}$, i.e.,
\begin{eqnarray}
&&{\cal R}_{v}(i)\equiv\left\{c_{v\cdot}\in R_{+}^{J_{v}}:\;h_{vk}(c_{v\cdot},i)\leq 0,\;k\in{\cal U}_{v}\right\}.
\elabel{capsur}
\end{eqnarray}
Furthermore, if we define $C_{U_{v}}(i)$ to be the sum capacity upper bound for ${\cal R}_{v}(i)$, the facet in the center of ${\cal O}_{v}(i)$ is linear and is assumed to be a non-degenerate $(J_{v}-1)$-dimensional region. More precisely, it can be represented by
\begin{eqnarray}
&&h_{vk_{U_{v}}}(c_{v\cdot},i)=\sum_{j\in{\cal J}(v)}c_{j}-C_{U_{v}}(i),
\elabel{hmiddle}
\end{eqnarray}
where, $k_{U_{v}}\in{\cal U}_{v}$ is the index corresponding to $C_{U_{v}}(i)$. In addition, we suppose that any one of the $J_{v}$ linear facets along the coordinate axes forms an $(J_{v}-1)$-user capacity region associated with a particular group of $J_{v}-1$ users if the queue corresponding to the other user is empty. In the same manner, we can provide an interpretation for the $(J_{v}-l)$-user capacity region for each $l\in\{2,...,J_{v}-1\}$.

Concerning the allocation of the service resources over the capacity regions to different users, we adopt the so-called head of line service discipline. Equivalently, the service goes to the packet at the head of the line for a serving queue where packets are stored in the order of their arrivals. The service rates are determined by a utility (or hash) function of the environmental state, the price for each user, and the number of packets in each of the queues. More precisely, for each state $i\in{\cal K}$, a price vector $p=(p_{1},...,p_{J})$, and a queue length vector $q=(q_{1},...,q_{J})'$, we define $\Lambda_{\cdot j}(pq,i)$ with $j\in{\cal J}$ to be the rate vector (in qubits/ps) of serving the $j$th queue at all its possible service pools, i.e.,
\begin{eqnarray}
&&\Lambda_{\cdot j}(pq,i)=c^{{\cal Q}(pq)}_{\cdot j}(i)=(c^{{\cal Q}(pq)}_{v_{1j}}(i),...,c^{{\cal Q}(pq)}_{v_{V_{j}j}}(i)),
\elabel{conventionI}
\end{eqnarray}
where,
\begin{eqnarray}
&&{\cal Q}(pq)\equiv\{j\in{\cal J},q_{j}=0\}.
\elabel{qindex}
\end{eqnarray}
Furthermore, let $\Lambda_{v\cdot}(pq,i)$ for each $v\in{\cal V}$ be the rate vector for all the users possibly served at service pool $v$, i.e.,
\begin{eqnarray}
&&\Lambda_{v\cdot}(pq,i)=c^{{\cal Q}(pq)}_{v\cdot}(i)=(c^{{\cal Q}(pq)}_{j_{v1}}(i),...,c^{{\cal Q}(pq)}_{j_{vJ_{v}}}(i)).
\elabel{conventionII}
\end{eqnarray}
Thus, $c^{{\cal Q}(pq)}_{v_{lj}}(i)=c^{{\cal Q}(pq)}_{j}(i)$ if the pool index $v_{lj}\in{\cal V}(j)$ for an integer $l\in\{1,...,V_{j}\}$ with $j\in{\cal J}$ while the total rate used in \eq{tjqalpha} can be represented by
\begin{eqnarray}
&&\Lambda_{j}(P(s)Q(s),\alpha(s))=\sum_{v\in{\cal V}(j)}c_{vj}^{{\cal Q}(P(s)Q(s))}(\alpha(s)).
\elabel{srates}
\end{eqnarray}
In the end, we impose the convention that an empty queue should not be served. Then, for each $v\in{\cal V}$ and ${\cal Q}\subseteq{\cal J}$ (e.g., a set as given by \eq{qindex}), we can define
\begin{eqnarray}
c^{{\cal Q}}_{j_{vl}}(i)&\equiv&\left\{\begin{array}{ll}
=0&\mbox{if}\;\;j_{vl}\in{\cal Q}\;\;\mbox{with}\;\;l\in\{1,...,J_{v}\},\\
>0&\mbox{if}\;\;j_{vl}\notin {\cal Q}\;\;\mbox{with}\;\;l\in\{1,...,J_{v}\},
\end{array}
\right.
\elabel{cqqvj}\\
\;\;\;\;\;\;\;c^{{\cal Q}}_{vj}(i)&\equiv&c^{{\cal Q}}_{j_{vl}}(i)\;\;\mbox{for some}\;\;j\in{\cal J}(v)\;\;\mbox{corresponding to each}\;\;
l\in\{1,...,J_{v}\},
\elabel{cqqvjI}\\
F^{v}_{{\cal Q}}(i)&\equiv&\bigg\{x\in{\cal R}_{v}(i):\;x_{j_{vl}}=0\;\;
\mbox{for all}\;\;j_{vl}\in{\cal Q}\;\;\mbox{with}\;\;
l\in\{1,...,J_{v}\}\bigg\}.
\elabel{fqindex}
\end{eqnarray}
Therefore, for all ${\cal Q}$ such that $\emptyset\subsetneqq{\cal Q}\subseteq{\cal J}(v)$ corresponding to each $v\in{\cal V}$ , if $c^{\cal Q}_{v\cdot}(i)$ is on the boundaries of the capacity region ${\cal R}_{v}(i)$, we have the following observation that
\begin{eqnarray}
&\left\{\begin{array}{ll}
\sum_{j\in{\cal J}(v)}c_{vj}^{\emptyset}(i)&\geq\;\;\;\sum_{j\in{\cal J}(v)}c_{vj}^{\cal Q}(i),\\
%\elabel{copproperty}\\
\sum_{j\in{\cal J}(v)\setminus{\cal Q}}c_{vj}^{\emptyset}(i)&\leq\;\;\;\sum_{j\in{\cal J}(v)\setminus{\cal Q}}c_{vj}^{\cal Q}(i),
%\elabel{coppropertyI}
\end{array}
\right.
\elabel{copproperty}
\end{eqnarray}
where, $c^{\emptyset}_{v\cdot}(i)\in{\cal O}_{v}(i)$ and $\emptyset$ denotes the empty set. Typical examples of our capacity region are referred to
the upper-left graph in Figure~\ref{numexample} for more details.

\subsubsection{A dynamic pricing and scheduling policy with users' selection}\label{apolicy}

For our purpose, we classify all the users into two types. More precisely, we first need to smartly choose the users to be served. In other words, at each time point and for each pool $v$, we intelligently select a set ${\cal M}(i,v)\equiv\{j_{v1}(i),...,j_{vM_{v}}(i)\}$ of users to get into services with $j_{vl}\in{\cal J}$ and $l\in\{1,...,M_{v}\}$ for a given positive integer number $M_{v}\leq J_{v}$. Among these chosen users, we need to conduct the dynamic pricing while realize optimal and fair resource allocation. Therefore, we design a strategy by mixing a saddle point and a static Pareto maximal-utility Nash equilibrium policy myopically at each time point $t$ to a mixed zero-sum and non-zero-sum game problem for each state $i\in{\cal K}$ and a given valued queue length vector $pq=(p_{1}q_{1},...,p_{J}q_{J})'$. Here we note that $p=(p_{1},...,p_{J})'$ is a given price vector and $q=(q_{1},...,q_{J})'$ is a given queue length vector such that $p_{j}=f_{j}(q_{j},i)$ as in \eq{pricefun} for each $j\in{\cal J}$ and $i\in{\cal K}$. The saddle point corresponds the users' selection while the Pareto optimality represents the full utilization of resources in the whole game system and the Nash equilibrium represents the fairness to all the chosen users. More exactly, in this game, there are $J$ users (players) associated with the $J$ queues. Each of them has his own utility function $U_{vj}(p_{j}q_{j},c_{vj})$ with $j\in{\cal J}(v)$ and $v\in{\cal V}(j)$. Every chosen user selects a policy to maximize his own utility function at each service pool $v$ while the summation of all the users' utility functions and the summation of the utility functions associated with the chosen users are also maximized. To wit, we can formulate a generalized users-selection, pricing, and resource-scheduling game problem by extending the ones in Examples~\ref{example1}-\ref{example2} as follows,
\begin{eqnarray}
&&\left\{\begin{array}{ll}
\max_{c_{v\cdot}\in F^{v}_{{\cal Q}}(i)}U_{00}(pq,c)&=\;\;\;U_{00}(pq,c^{*}(i)),\\
%\elabel{gameoptoo}\\
\max_{c_{v\cdot}\in F^{v}_{{\cal Q}}(i)}U_{0j}(pq,c)&=\;\;\;U_{0j}(pq,c^{*}(i)),\;j\in{\cal M}(i,v)\cap({\cal J}(v)\setminus{\cal Q}(q)),\\
%\elabel{gameoptoI}\\
\max_{c_{v\cdot}\in F^{v}_{{\cal Q}}(i)}(-U_{0j}(pq,c))&=\;\;\;-U_{0j}(pq,c^{*}(i)),\;j\in({\cal J}(v)\setminus{\cal Q}(q))\setminus{\cal M}(i,v))
%\elabel{gameoptoII}
\end{array}
\right.
\elabel{gameoptoo}
\end{eqnarray}
while we have that
\begin{eqnarray}
&&\left\{\begin{array}{ll}
\max_{c_{v\cdot}\in F^{v}_{{\cal Q}}(i),\;j\in{\cal M}(i,v)\bigcap\left({\cal J}(v)\setminus{\cal Q}(q)\right)}U_{vj}(pq,c)&=\;\;\;U_{vj}(pq,c^{*}(i)),\\
%\elabel{gameoptoIII}\\
\max_{c_{v\cdot}\in F^{v}_{{\cal Q}}(i),\;j\in\left({\cal J}(v)\setminus{\cal Q}(q)\right)\setminus{\cal M}(i,v)}(-U_{vj}(pq,c))&=\;\;\;-U_{vj}(pq,c^{*}(i)).
%\elabel{gameoptoIV}
\end{array}
\right.
\elabel{gameoptoIV}
\end{eqnarray}
Note that, the rate vector $c$ in \eq{gameoptoo}-\eq{gameoptoIV} is given by
\begin{eqnarray}
&&c=((c_{j_{11}},...,c_{j_{1J_{1}}}),...,(c_{j_{V1}},...,c_{j_{VJ_{V}}}))
\nonumber
\end{eqnarray}
and the utility functions used in \eq{gameoptoo}-\eq{gameoptoIV} are defined by
\begin{eqnarray}
&\left\{\begin{array}{ll}
U_{00}(pq,c)&=\;\;\;\sum_{j\in{\cal J}(v)\setminus{\cal Q}(q)}\sum_{v\in{\cal V}(j)}U_{vj}(p_{j}q_{j},c_{vj}),\\
%\elabel{gameoptoI}
U_{0j}(pq,c)&=\;\;\;\sum_{v\in{\cal V}(j)}U_{vj}(p_{j}q_{j},c_{vj})
\;\;\;\;\;\;\;\;\;\;\;\;\;\;\;\;\;\;\;\;\;\;\;
\mbox{for each}\;\;j\in{\cal J}(v)\setminus{\cal Q}(q),\\
%\elabel{gameoptoI}
U_{vj}(pq,c)&=\;\;\;U_{vj}(p_{j}q_{j},c_{vj})
\;\;\;\;\;\;\;\;\;\;\;\;\;\;\;\;\;
\mbox{for each}\;\;j\in{\cal J}(v)\setminus{\cal Q}(q)\;\;
\mbox{and}\;\;v\in{\cal V}(j).
\end{array}
\right.
\nonumber
\end{eqnarray}
Then, by extending the concepts of Nash equilibrium and Pareto optimality in Dai~\cite{dai:plamod,dai:quacom}, Nash~\cite{nas:equpoi} and Rosen~\cite{ros:exiuni}, we have the following definition concerning a mixed saddle point and static Pareto maximal-utility Nash equilibrium policy myopically at each particular time point for the users' selection, dynamic pricing, and resource scheduling as follows.
\begin{definition}\label{definitiona}
For each state $i\in{\cal K}$, a price vector $p\in R^{J}_{+}$, and a queue length vector $q\in R^{J}_{+}$ such that $\eq{pricefun}$ is satisfied, we call the rate vector
\begin{eqnarray}
&&c^{*}(i)\in F_{{\cal Q}(q)}(i)\equiv F^{1}_{{\cal Q}(q)}(i)\times...\times F^{V}_{{\cal Q}(q)}(i)
\nonumber
\end{eqnarray}
a mixed saddle-point and static Pareto maximal-utility Nash equilibrium policy to the mixed zero-sum and non-zero-sum game problem
in \eq{gameoptoo}-\eq{gameoptoIV} if, for each $j\in{\cal J}(v)\setminus{\cal Q}(q)$ and any given $c(i)\in F_{{\cal Q}(q)}(i)$, the following facts are true,
\begin{eqnarray}
&\left\{\begin{array}{ll}
U_{00}(pq,c^{*}(i))&\geq\;\;\;U_{00}(pq,c(i)),\\
%\elabel{nequilibrium}\\
%U_{0j}(q,c^{*}(i))&\geq& U_{0j}(q,c^{*}_{\cdot-j}(i)),
%\elabel{nequilibriumI}\\
U_{vj}(pq,c^{*}(i))&\geq\;\;\;U_{vj}(q,c^{*}_{\cdot-j}(i))\;\;\;\mbox{if}\;\;j\in{\cal M}(i,v)\cap({\cal J}(v)\setminus{\cal Q}(q)),\;v\in\{0\}\cup{\cal V}(j),\\
%\elabel{nequilibriumII}\\
-U_{vj}(pq,c^{*}(i))&\geq\;\;\;-U_{vj}(pq,c^{*}_{\cdot-j}(i))\;\;\mbox{if}\;\;j\in({\cal J}(v)\setminus{\cal Q}(q))\setminus{\cal M}(i,v),\;v\in\{0\}\cup{\cal V}(j),\\
%\elabel{nequilibriumIV}\\
c^{*}_{\cdot -j}(i)&\equiv\;\;\;(c_{\cdot 1}^{*}(i),...,c^{*}_{\cdot j-1}(i),c_{\cdot j}(i),c^{*}_{\cdot j+1}(i),...,c^{*}_{\cdot J}(i)).
%\elabel{nequilibriumIII}
\end{array}
\right.
\nonumber
\end{eqnarray}
\end{definition}

\subsection{RDRS modeling under the policy}

Before stating our main theorem, we first introduce another concept of the so-called mixed saddle point and static Pareto minimal-dual-cost Nash equilibrium policy myopically at each given time point for a given price parameter $p\in R_{+}^{J}$. Then, based on the mixed policy, we can inversely obtain the price vector and determine the target rate vector. To do so, we formulate a mixed users-selection, pricing and minimal-dual-cost game problem associated with the mixed game problem in \eq{gameoptoo}-\eq{gameoptoIV}. More precisely, for a given $i\in{\cal K}$, a price parameter $p\in R_{+}^{J}$, a rate vector $c\in {\cal R}(i)\equiv{\cal R}_{1}(i)\times...\times {\cal R}_{V}(i)$, and a parameter $w\geq 0$, the mixed minimal-dual-cost problem can be presented as follows:
%\begin{eqnarray}
%&&\min_{p\in R^{J}_{+}}\left\{\min_{q\in R^{J}_{+}}C_{00}(pq,c)\right\},
%\elabel{costminpo}\\
%&&\min_{p_{j}\in R_{+},\;j\in{\cal M}(i,v)\bigcap\left({\cal C}(c)\bigcap{\cal J}(v)\right)}\left\{\min_{q_{j}\in R_{+},\;j\in{\cal %M}(i,v)\bigcap\left({\cal C}(c)\bigcap{\cal J}(v)\right)}C_{vj}(pq,c)\right\},
%\elabel{costminp}\\
%&&\min_{p_{j}\in R_{+},\;j\in\left({\cal C}(c)\bigcap{\cal J}(v)\right)\setminus{\cal M}(i,v)}\left\{\min_{q_{j}\in R_{+},\;j\in\left({\cal %C}(c)\bigcap{\cal J}(v)\right)\setminus{\cal M}(i,v)}(-C_{vj}(pq,c))\right\}
%\elabel{costminpI}
%\end{eqnarray}
\begin{eqnarray}
&&\left\{\begin{array}{ll}
\min_{q\in R^{J}_{+}}C_{00}(pq,c),\\
%\elabel{costminpo}\\
\min_{q_{j}\in R_{+},\;j\in{\cal M}(i,v)\bigcap\left({\cal C}(c)\bigcap{\cal J}(v)\right)}C_{vj}(pq,c),\\
%\elabel{costminp}\\
\min_{q_{j}\in R_{+},\;j\in\left({\cal C}(c)\bigcap{\cal J}(v)\right)\setminus{\cal M}(i,v)}(-C_{vj}(pq,c))
%\elabel{costminpI}
\end{array}
\right.
\elabel{costminpo}
\end{eqnarray}
subject to
\begin{eqnarray}
\sum_{j\in{{\cal M}(i,v)\bigcap\cal C}(c)}\frac{q_{j}}{\mu_{j}}&\geq& w,
\nonumber
%&&\;\;\;\;\;\;\;q_{j}\geq 0\;\;\;\;\mbox{for each}\;\;\;j\in{\cal J},
\end{eqnarray}
where, the cost function $C_{vj}(pq,c)$ for each $j\in{\cal J}(v)$ and $v\in\{0\}\cup{\cal V}(j)$ is defined by
\begin{eqnarray}
&\left\{\begin{array}{ll}
C_{00}(pq,c)&=\;\;\;\sum_{j\in{\cal C}(c)\bigcap{\cal J}(v)}\sum_{v\in{\cal V}(j)}C_{j}(p_{j}q_{j},c_{vj}),\\
%\elabel{vcost}\\
%\nonumber\\
%\nonumber\\
C_{0j}(pq,c)&=\;\;\;\sum_{v\in{\cal V}(j)}C_{vj}(p_{j}q_{j},c_{vj}),\\
%\elabel{vcost}\\
%\nonumber\\
%\nonumber\\
C_{vj}(pq,c)&=\;\;\;C_{vj}(p_{j}q_{j},c_{vj})=\frac{1}{\mu_{j}}\int_{0}^{q_{j}}\frac{\partial U_{vj}(p_{j}u,c_{vj})}{\partial c_{vj}}du\;\;\mbox{for}
\;\;j\in{\cal C}(c)\cap{\cal J}(v),\;v\in{\cal V}(j),
%\elabel{costf}
\end{array}
\right.
\nonumber
\end{eqnarray}
and ${\cal C}(c)$ is an index set associated with the non-zero rates and non-empty queues, i.e.,
\begin{eqnarray}
&&{\cal C}(c)\equiv\bigg\{j:c_{\cdot j}\neq 0\;\;\mbox{componentwise with}\;\;j\in{\cal J}\bigg\}.
\nonumber
\end{eqnarray}
In other words, if the environment is in state $i\in{\cal K}$, we try to find a queue state $q$ for a given $c\in{\cal R}(i)$, a price parameter vector $p\in R_{+}^{J}$, and a given parameter $w\geq 0$ such that the individual user's dual-costs and the total dual-cost over the system are all minimized at the same time while the (average) workload meets or exceeds $w$. Then, we have the following definitions.
%Similar to Definition~\ref{definitiona},
%we have the following definition concerning a Pareto minimal-dual-cost
%Nash equilibrium policy.}}
\begin{definition}\label{definitionb}
%Suppose that $(p^{*},q^{*})\in R^{J}_{+}\times R^{J}_{+}$ satisfies the Lipschitz condition in \eq{pricefun} and $p^{*}_{j}q^{*}_{j}=0$ if $j\in{\cal %J}\setminus {\cal C}(c)$. Then,
For each state $i\in{\cal K}$, a price vector $p\in R_{+}^{J}$, and a rate vector $c(i)\in{\cal R}(i)$, a queue length vector $q^{*}\in R^{J}_{+}$ is called a mixed saddle point and static Pareto minimal-dual-cost Nash equilibrium policy to the mixed zero-sum and non-zero-sum game problem in \eq{costminpo} if, for each $j\in{\cal C}(c)$, $v\in\{0\}\cup{\cal V}$, and any given $q\in R^{J}_{+}$ with $q_{j}=0$ when $j\in{\cal J}\setminus {\cal C}(c)$, we have that
\begin{eqnarray}
&&\left\{\begin{array}{ll}
C_{00}(pq^{*},c(i))&\leq\;\;\;C_{00}(pq,c(i)),\\
%\elabel{cnequilibrium}\\
C_{vj}(pq^{*},c(i))&\leq\;\;\;C_{vj}(pq^{*}_{-j},c(i))\;\;\mbox{if}\;\;j\in{\cal M}(i,v)\cap({\cal C}(c)\cap{\cal J}(v)),\\
%\elabel{cnequilibriumI}\\
-C_{vj}(pq^{*},c(i))&\leq\;\;\;-C_{vj}(pq^{*}_{-j},c(i))\;\;\mbox{if}\;\;j\in({\cal C}(c)\cap{\cal J}(v))\setminus{\cal M}(i,v),\\
%\elabel{cnequilibriumIo}\\
%p^{*}_{-j}&\equiv\;\;\;(p_{1}^{*},...,p^{*}_{j-1},p_{j},p^{*}_{j+1},...,p^{*}_{J}),\\
%\elabel{cnequilibriumIa}\\
q^{*}_{-j}&\equiv\;\;\;(q_{1}^{*},...,q^{*}_{j-1},q_{j},q^{*}_{j+1},...,q^{*}_{J}).
%\elabel{cnequilibriumI}
\end{array}
\right.
\elabel{cnequilibrium}
\end{eqnarray}
%\begin{eqnarray}
%&&\left\{\begin{array}{ll}
%C_{00}(p^{*}q^{*},c(i))&\leq\;\;\;C_{00}(pq,c(i)),\\
%%\elabel{cnequilibrium}\\
%C_{vj}(p^{*}q^{*},c(i))&\leq\;\;\;C_{vj}(p^{*}_{-j}q^{*}_{-j},c(i))\;\;\mbox{if}\;\;j\in{\cal M}(i,v)\cap({\cal C}(c)\cap{\cal J}(v)),\\
%%\elabel{cnequilibriumI}\\
%-C_{vj}(p^{*}q^{*},c(i))&\leq\;\;\;-C_{vj}(p^{*}_{-j}q^{*}_{-j},c(i))\;\;\mbox{if}\;\;j\in({\cal C}(c)\cap{\cal J}(v))\setminus{\cal M}(i,v),\\
%%\elabel{cnequilibriumIo}\\
%p^{*}_{-j}&\equiv\;\;\;(p_{1}^{*},...,p^{*}_{j-1},p_{j},p^{*}_{j+1},...,p^{*}_{J}),\\
%%\elabel{cnequilibriumIa}\\
%q^{*}_{-j}&\equiv\;\;\;(q_{1}^{*},...,q^{*}_{j-1},q_{j},q^{*}_{j+1},...,q^{*}_{J}).
%%\elabel{cnequilibriumI}
%\end{array}
%\right.
%\elabel{cnequilibrium}
%\end{eqnarray}
\end{definition}

Note that, once we obtain the queue policy point $q^{*}$ with respect to the given price vector $p$ from Definition~\ref{definitionb}, we can inversely deduce the corresponding price policy vector $p$ in terms of $q^{*}$, i.e., $p=g(q^{*})$ as in \eq{pricefun}. This relationship can be used to design iterative algorithms in our numerical simulations. Furthermore, in Definition~\ref{definitionb}, we have used more strict concept of ``Pareto optimal Nash equilibrium point", this concept can be relaxed to ``Pareto optimal point" and the related theoretical discussion keeps true. In certain cases and when it is necessary, we can shift the Pareto optimal point to the Pareto optimal Nash equilibrium point by some mapping techniques.
\begin{definition}
Let $\hat{Q}^{r,(P,G)}(\cdot)$ and $\hat{W}^{r,(P,G)}(\cdot)$ be the diffusion-scaled queue length and workload processes respectively under an arbitrarily feasible dynamic pricing and rate scheduling policy $(P,G)$ satisfying the Lipschitz condition in \eq{pricefun}. A vector process $\hat{Q}(\cdot)$ is called a mixed asymptotic saddle point and Pareto minimal-dual-cost Nash equilibrium policy globally over the whole time horizon if, for any $t\geq 0$ and $v\in\{0\}\cup{\cal V}(j)$ with $j\in{\cal J}$, we have that
\begin{eqnarray}
&&\liminf_{r\rightarrow\infty}C_{00}(P(t)\hat{Q}^{r,(P,G)}(t),\rho_{j}(\alpha(t)))\geq C_{00}(P(t)\hat{Q}(t),\rho_{j}(\alpha(t))).
\elabel{optimaleqno}
\end{eqnarray}
Furthermore, for each $j\in{\cal M}(\alpha(t),v,t)\cap({\cal C}(c)\cap{\cal J}(v))$, we have that
\begin{eqnarray}
&&\liminf_{r\rightarrow\infty}C_{vj}(P(t)\hat{Q}_{-j}^{r,(P,G)}(t),\rho_{j}(\alpha(t)))\geq C_{vj}(P(t)\hat{Q}(t),\rho_{j}(\alpha(t))).
\elabel{optimaleqnI}
\end{eqnarray}
In addition, for each $j\in({\cal C}(c)\cap{\cal J}(v))\setminus{\cal M}(\alpha(t),v,t)$, we have that
\begin{eqnarray}
&&\liminf_{r\rightarrow\infty}\left(-C_{vj}(P(t)\hat{Q}_{-j}^{r,(P,G)}(t),\rho_{j}(\alpha(t)))\right)\geq -C_{vj}(P(t)\hat{Q}(t),\rho_{j}(\alpha(t))).
\elabel{optimaleqnIo}
\end{eqnarray}
Note that, in \eq{optimaleqnI}-\eq{optimaleqnIo} and for each $j\in{\cal J}$, we have that
\begin{eqnarray}
%P_{-j}(t)&=&(\hat{P}_{1}(t),...,\hat{P}_{j-1}(t),P_{j}(t),\hat{P}_{j+1}(t),...,\hat{P}_{J}(t)),
%\elabel{adaqho}\\
\hat{Q}_{-j}^{r,(P,G)}(t)&=&(\hat{Q}_{1}(t),...,\hat{Q}_{j-1}(t),\hat{Q}_{j}^{r,(P,G)}(t),\hat{Q}_{j+1}(t),...,\hat{Q}_{J}(t)).
\elabel{adaqh}
\end{eqnarray}
\end{definition}

Next, let $q^{*}(w,p,\rho(i))$ be the mixed saddle point and Pareto minimal-dual-cost Nash equilibrium policy to the game problem in \eq{costminpo} in terms of each given number $w\geq 0$, $p\in R_{+}^{J}$, and $i\in{\cal K}$ at a given time $t$. Furthermore, let $p(w,q^{*},\rho(i))$ denote its corresponding inverse price vector with respect to $q^{*}$ and construct price policy vector
\begin{eqnarray}
&&p^{*}(w,q^{*},\rho(i))=f(p(w,q^{*},\rho(i)))
\elabel{policyv}
\end{eqnarray}
such that the Lipschitz condition in \eq{pricefun} is satisfied. Then, our main theorem can be presented as follows.
\begin{theorem}\label{qdiff}
For the game-competition based users' selection, dynamic pricing, and scheduling policy determined by \eq{gameoptoo}-\eq{costminpo} and \eq{policyv} with $Q^{r}(0)=0$ and conditions \eq{heavytrafficc}-\eq{intercon} (that will be detailed in Section~\ref{rdrsm}), we have that
\begin{eqnarray}
&&(\hat{Q}^{r}(\cdot),\hat{W}^{r}(\cdot))\Rightarrow(\hat{Q}(\cdot),\hat{W}(\cdot))\;\;\;\mbox{along}\;\;\;r\in{\cal R},
\elabel{I-qwyweakc}
\end{eqnarray}
where, ``$\Rightarrow$" denotes ``convergence in distribution". Furthermore, the limit queue length $\hat{Q}(\cdot)$ and the total workload $\hat{W}(\cdot)$ have the relationship
\begin{eqnarray}
&&\left\{\begin{array}{ll}
\hat{Q}(t)&=\;\;\;q^{*}(\hat{W}(t),\hat{P}(t),\rho(\alpha(t))),\\
\hat{P}(t)&=\;\;\;p^{*}(\hat{W}(t),\hat{Q}(t),\rho(\alpha(t))),
\end{array}
\right.
\elabel{qqast}
\end{eqnarray}
where, $\hat{P}(\cdot)$ the inverse price vector process defined through \eq{policyv} and $\hat{W}(\cdot)$ is a $1$-dimensional RDRS in strong sense with
\begin{eqnarray}
&&\left\{\begin{array}{ll}
b(i,t)&=\;\;\;\sum_{j\in\bigcup_{v\in{\cal V}}{\cal M}(i,v,t)}\frac{\theta_{j}(i)}{\mu_{j}},\\
%b(i,t)&=&\;\frac{\theta_{j_{1}}(i)}{\mu_{j_{1}(1)}}+\cdot\cdot\cdot+\frac{\theta_{j_{M}}(i)}{\mu_{j_{M}}}\;\;\;
%\;\;\;\;\;\;\;\;\;\mbox{for}\;\;j_{l}\in{\cal M}(i,t), l\in\{1,...,M\},
%\elabel{twocI}\\
\sigma^{E}(t)&=\;\;\;\sigma^{S}(t)=\left(\hat{\sigma}_{1}(t),...,\hat{\sigma}_{J}(t)\right),\\
%\sigma^{E}(t)&=&\;\sigma^{S}(t)=\left(1/\mu_{j_{1}},...,1/\mu_{j_{M}}\right)£¬
%\elabel{twocII}\\
\hat{\sigma}_{j}(t)&=\;\;\;\left\{\begin{array}{ll}
\frac{1}{\mu_{j}}&\mbox{if}\;\;j\in\bigcup_{v\in{\cal V}}{\cal M}(i,v,t),\\
0&\mbox{otherwise,}
\end{array}
\right.\\
%\nonumber\\
%\sigma^{S}&=&\left(\frac{1}{\mu_{1}},...,\frac{1}{\mu_{J}}\right),
%\elabel{twocIII}\\
%\nonumber\\
R(i,t)&=\;\;\;1
%\elabel{twocIV}
\end{array}
\right.
\elabel{twocI}
\end{eqnarray}
for $t\in[0,\infty)$ and some constant $\theta_{j}(i)$ for each $j\in\bigcup_{v\in{\cal V}}{\cal M}(i,v,t)$. In addition, there is a common supporting probability space, under which and with probability one, the limit queue length $\hat{Q}(\cdot)$ is an asymptotic mixed saddle point and Pareto minimal-dual-cost Nash equilibrium policy globally over time interval $[0,\infty)$. Finally, the limit workload $\hat{W}(\cdot)$ is also asymptotic minimal in the sense that
\begin{eqnarray}
&&\liminf_{r\rightarrow\infty}\hat{W}^{r,(P,G)}(t)\geq\hat{W}(t).
\elabel{optimaleqn}
\end{eqnarray}
\end{theorem}

The proof of Theorem~\ref{qdiff} will be given in Section~\ref{rdrsm}. To illustrate the efficiency of our model under the users' selection, dynamic pricing, and resource scheduling policy, we first present simulation case studies in the following subsection. More precisely, for a constant $T\in[0,\infty)$, we divide the interval $[0,T]$ equally into $n$ subintervals $\{[t_{i},t_{i+1}],i\in\{0,1,...,n-1\}\}$ with $t_{0}=0$, $t_{n}=T$, and $\Delta t_{i}=t_{i+1}-t_{i}=\frac{T}{n}$. Furthermore, let
\begin{eqnarray}
&&\Delta F(t_{i})\equiv F(t_{i})-F(t_{i-1})
\elabel{difpro}
\end{eqnarray}
for each process $F(\cdot)\in\{B^{E}(\cdot),B^{S}(\cdot),\hat{W}(\cdot),\hat{Y}(\cdot)\}$. Then, we can develop an iterative procedure
%similar to those as in Dai~\cite{dai:plamod,dai:quacom}
to simulate the RDRS model under our policy that is derived in Theorem~\ref{qdiff}.

\subsection{Simulation case studies via RDRS models}\label{singlepoolexample}

In this subsection, we conduct simulation case studies for Examples~\ref{example1}-\ref{example2} presented in Subsubsection~\ref{illex}. The main point of these simulation studies is to illustrate our policies proposed in the two examples outperform several policies in certain ways. These policies used for the purpose of comparisons include an existing constant pricing policy, an existing 2D-Queue policy, a newly designed randomly users' selection stochastic pooling policy, and an arbitrarily selected dynamic pricing policy. As mentioned in Subsubsection~\ref{illex}, Examples~\ref{example1}-\ref{example2} are corresponding to a single-pool system with two-users and three-users respectively. Thus, we will omit all the related pool index $v$. In an associated real-world system, the parameter vectors $p$ and $q$ in \eq{adafn} (or \eq{adaf}) are the randomly evolving pricing process $P(t)$ in \eq{pricefun} and the queue length process $Q(t)$ in \eq{queuelength}. Hence, it is our concern of this subsection about how to employ the RDRS performance model in Definition~\ref{rdrsd} to evaluate the usefulness of our proposed myopic users' selection, dynamic pricing, and scheduling policies globally over the whole time horizon $[0,\infty)$ for Examples~\ref{example1}-\ref{example2}. To interpret our numerical simulation implementations, we first identify the corresponding dual-cost functions {$C_{j}(q,c)$ as defined in \eq{costminpo} with $j\in\{1,2,3\}$ for the associated $U_{j}(q,c)$ given in \eq{numutility} and \eq{anumutility}. More precisely,
\begin{eqnarray}
&&\left\{\begin{array}{ll}
C_{1}(p_{1}q_{1},c_{1})&=\;\;\frac{1}{\mu_{1}}\int_{0}^{q_{1}}\frac{\partial U_{1}(p_{1}u,c_{1})}{\partial c_{1}}du\;\;=\;\;\frac{(p_{1}q_{1})^{2}}{2\mu_{1}c_{1}},\\
%\elabel{numcosto}\\
C_{2}(p_{2}q_{2},c_{2})&=\;\;\frac{1}{\mu_{2}}\int_{0}^{q_{2}}\frac{\partial U_{2}(p_{2}u,c_{2})}{\partial c_{2}}du\;\;=\;\;\frac{2(p_{2}q_{2})^{3}}{3\mu_{2}c^{3}_{2}},\\
%\elabel{numcost}\\
C_{3}(p_{3}q_{3},c_{3})&=\;\;\frac{1}{\mu_{3}}\int_{0}^{q_{3}}\frac{\partial U_{3}(p_{3}u,c_{3})}{\partial c_{3}}du\;\;=\;\;\frac{(p_{3}q_{3})^{2}}{2\mu_{3}c_{3}},
%\elabel{numcostI}
\end{array}
\right.
\elabel{numcosto}
\end{eqnarray}
where, $1/\mu_{j}$ for all $j\in\{1,2,3\}$ are average quantum packet lengths associated with the three users as explained just before \eq{sjvc}.

\subsubsection{The Simulation for Example~\ref{example1}}

Based on the first two dual-cost functions in \eq{numcosto}, we can formulate a corresponding 2-stage minimal dual-cost non-zero-sum game problem for a price parameter $p\in R_{+}^{2}$ as follows,
\begin{eqnarray}
&&\min_{q\in R_{+}^{2}}C_{j}(pq,c)\;\;\mbox{subject to}\;\;\frac{q_{1}}{\mu_{1}}+\frac{q_{2}}{\mu_{2}}\geq w
\elabel{adafcn}
\end{eqnarray}
for a fixed constant $w>0$, a fixed $c\in{\cal R}$, and all $j\in\{0,1,2\}$ with $C_{0}(pq,c)=C_{1}(pq,c)+C_{2}(pq,c)$. Since $C_{j}(p_{j}q_{j},c_{j})$ for each $j\in\{1,2\}$ is strictly increasing with respect to $p_{j}q_{j}$ (or simply $q_{j}$), a Pareto minimal dual-cost Nash equilibrium point to the problem in \eq{adafcn} must be located on the line where the equality of the constraint inequality is true (i.e., $q_{1}/\mu_{1}+q_{2}/\mu_{2}=w$). Thus, we know that
\begin{eqnarray}
&&q_{j}=\mu_{j}\left(w-\frac{q_{2-j+1}}{\mu_{2-j+1}}\right)\;\;\;\mbox{with}\;\;\;j\in\{1,2\}.
\elabel{qjji}
\end{eqnarray}
Hence, it follows from \eq{qjji} that
\begin{eqnarray}
&&\bar{f}(q_{1})\equiv\sum_{j=1}^{2}C_{j}(p_{j}q_{j},c_{j})
%=\frac{q^{2}_{1}}{2\mu_{1}c_{1}}+\frac{2\mu_{2}^{2}}{3c_{2}^{3}}\left(w-\frac{q_{1}}{\mu_{1}}\right)^{3}.
%\elabel{qjjiI}\\
%&=&\frac{q^{2}_{1}}{2(\mu_{1}/p_{1}^{2})c_{1}}+\frac{2(\mu_{2}/p_{2}^{3})^{2}}{3c_{2}^{3}}\left(w-\frac{q_{1}}{(\mu_{1}/p_{1}^{2})}\right)^{3}.
%\nonumber\\
=\frac{p_{1}^{2}q^{2}_{1}}{2\mu_{1}c_{1}}+\frac{2p_{2}^{3}\mu_{2}^{2}}{3c_{2}^{3}}\left(w-\frac{q_{1}}{\mu_{1}}\right)^{3}.
\elabel{qjjiI}
\end{eqnarray}
%\begin{eqnarray}
%&&\bar{f}(q_{1})\equiv\sum_{j=1}^{2}C_{j}(p_{j}q_{j},c_{j})
%%=\frac{q^{2}_{1}}{2\mu_{1}c_{1}}+\frac{2\mu_{2}^{2}}{3c_{2}^{3}}\left(w-\frac{q_{1}}{\mu_{1}}\right)^{3}.
%%\elabel{qjjiI}\\
%%&=&\frac{q^{2}_{1}}{2(\mu_{1}/p_{1}^{2})c_{1}}+\frac{2(\mu_{2}/p_{2}^{3})^{2}}{3c_{2}^{3}}\left(w-\frac{q_{1}}{(\mu_{1}/p_{1}^{2})}\right)^{3}.
%%\nonumber\\
%=\frac{p_{1}^{2}q^{2}_{1}}{2\mu_{1}c_{1}}+\frac{2\mu_{2}^{2}}{3p_{2}^{6}c_{2}^{3}}\left(w-\frac{p_{1}^{2}q_{1}}{\mu_{1}}\right)^{3}.
%\elabel{qjjiI}
%\end{eqnarray}
Then, by solving the equation $\frac{\partial\bar{f}(q_{1})}{\partial q_{1}}=0$, we can get the minimal value of the function $\bar{f}(q_{1})$ for each $p\in R_{+}^{2}$. More precisely, the unique Pareto minimal point $q^{*}(p,w)=(q_{1}^{*},q_{2}^{*})(p,w)$ to the problem in \eq{adafcn} can be explicitly given by
\begin{eqnarray}
&&\;\;\;\;\left\{\begin{array}{ll}
q_{1}^{*}(p,w)&\equiv\;\;\bar{g}_{1}(p,w)\;=\;\frac{1}{2}\left(\frac{2w}{\mu_{1}}
+\frac{p_{1}^{2}c_{2}^{3}}{2p_{2}^{3}c_{1}\mu_{2}^{2}}\right)\mu_{1}^{2}-\sqrt{\frac{1}{4}\left(\frac{2w}{\mu_{1}}
+\frac{p_{1}^{2}c_{2}^{3}}{2p_{2}^{3}c_{1}\mu_{2}^{2}}\right)^{2}\mu_{1}^{4}-\mu_{1}^{2}w^{2}},\\
%\elabel{explictequil}\\
q_{2}^{*}(p,w)&\equiv\;\;\bar{g}_{2}(p,q_{1}^{*}(p,w),w)\;=\;\mu_{2}\left(w-\frac{q_{1}^{*}(p,w)}{\mu_{1}}\right).
\end{array}
\right.
\elabel{explictequilI}
\end{eqnarray}
%\begin{eqnarray}
%&&\left\{\begin{array}{ll}
%q_{1}^{*}(p,w)&\equiv\;\;\bar{g}_{1}(p,w)\;=\;\frac{1}{2}\left(\frac{2p^{2}_{1}w}{\mu_{1}}
%+\frac{p_{2}^{6}c_{2}^{3}}{2c_{1}\mu_{2}^{2}}\right)\frac{\mu_{1}^{2}}{p_{1}^{4}}-\sqrt{\frac{1}{4}\left(\frac{2p_{1}^{2}w}{\mu_{1}}
%+\frac{p_{2}^{6}c_{2}^{3}}{2c_{1}\mu_{2}^{2}}\right)^{2}\frac{\mu_{1}^{4}}{p_{1}^{8}}-\frac{\mu_{1}^{2}w^{2}}{p_{1}^{4}}},\\
%%\elabel{explictequil}\\
%q_{2}^{*}(p,w)&\equiv\;\;\bar{g}_{2}(p,w)\;=\;\frac{\mu_{2}}{p_{2}^{3}}\left(w-\frac{p_{1}^{2}q_{1}^{*}(p,w)}{\mu_{1}}\right),
%\end{array}
%\right.
%\elabel{explictequilI}
%\end{eqnarray}
\begin{figure}[tbh]
%\centerline{\epsfxsize=7.0in\epsfbox{twopricingp.eps}}
\centerline{\epsfxsize=7.1in\epsfbox{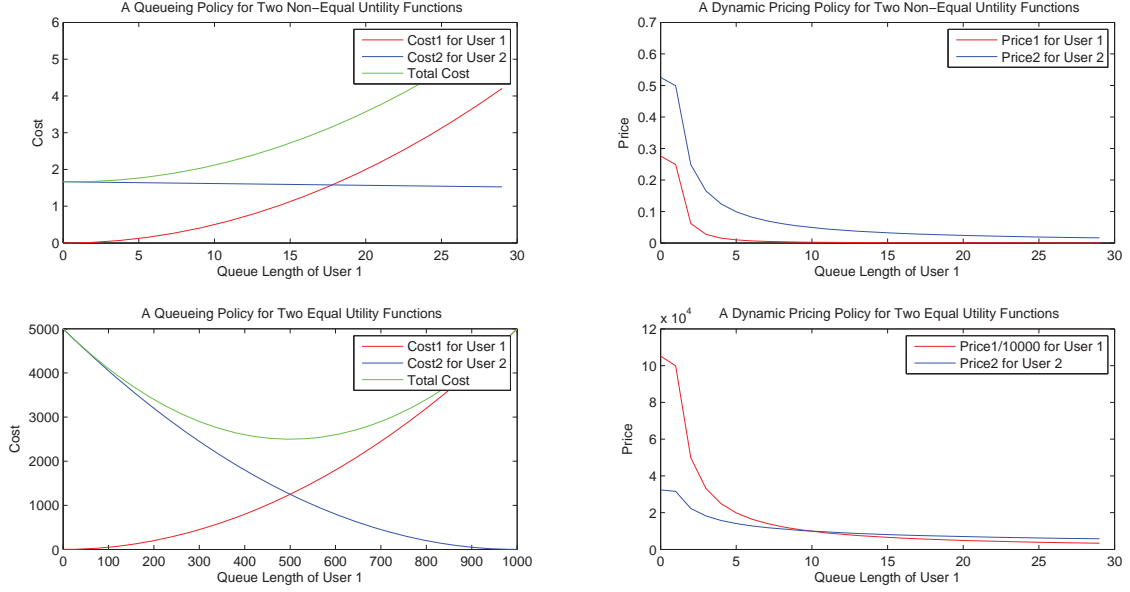}}
\caption{\small Pareto optimal Nash equilibrium policies with dynamic pricing}
\label{numexamplei}
%\centerline{\epsfxsize=2.0in\epsfbox{paretooptimal.eps}}
%\caption{\small An example of Pareto optimal Nash
%equilibrium policy}
\end{figure}
From the green curve in the upper-left graph of Figure~\ref{numexamplei} where $p_{1}=p_{2}=1$ and $w=10000$, we can see that this point is close to the one corresponding to $q_{1}=0$ and we can consider it as a Pareto minimal Nash equilibrium point near boundary. Another Nash equilibrium point is the intersection point of the red and blue curves in the left graph of Figure~\ref{numexamplei}. Obviously, this point is not a minimal total cost point. However, we can use some transformation technique to shift the minimal point to this one and to design a more fairly balanced decision policy. Nevertheless, for the purpose of this research in finding the Pareto utility-maximization Nash equilibrium policy, we use the point in \eq{explictequilI} as our decision policy. In this case, for the price parameter $p\in R_{+}^{2}$, we have
\begin{eqnarray}
&&\left\{\begin{array}{ll}
C_{0}(pq^{*},c)\leq C_{0}(pq,c), &\\
%\elabel{fdualcI}\\
C_{1}(pq^{*},c)\leq C_{1}(pq_{-1}^{*},c)&\mbox{with}\;\;\;q^{*}_{-1}=(q_{1},q_{2}^{*}),\\
%\elabel{fdualcII}\\
C_{2}(pq^{*},c)\leq C_{2}(pq_{-2}^{*},c)&\mbox{with}\;\;\;q^{*}_{-2}=(q_{1}^{*},q_{2}).
%\elabel{fdualcIII}
\end{array}
\right.
\elabel{fdualcI}
\end{eqnarray}
Then, associated with a given queue length based Pareto minimal Nash equilibrium point in \eq{explictequilI}, we can obtain the relationship between prices $p_{1}$ and $p_{2}$ as follows,
\begin{eqnarray}
&&\frac{p^{2}_{1}(q_{1},w)}{p_{2}^{3}(q_{1},w)}\equiv\kappa(q_{1},w)=\left(\frac{2c_{1}\mu_{2}^{2}}{c_{2}^{3}}\right)
\left(\frac{\mu_{1}^{2}w^{2}+q_{1}^{2}}{\mu_{1}^{2}q_{1}}-\frac{2w}{\mu_{1}}\right).
\elabel{pqexplictequilI}
\end{eqnarray}
From \eq{pqexplictequilI}, we can see that there are different choices of dynamic pricing policies corresponding to Pareto minimal Nash equilibrium point $q^{*}(p,w)$ in \eq{explictequilI}. For the current study, we take
\begin{eqnarray}
&&\;\;\;\;\left\{\begin{array}{ll}
p_{1}(q_{1},w)&=\;\;\;\kappa^{2}(q_{1},w),\\
p_{2}(q_{1},w)&=\;\;\;\kappa(q_{1},w),
\end{array}
\right.
\elabel{dp1p2}
\end{eqnarray}
whose dynamic evolutions with the queue length $q_{1}$ are shown in the upper-right graph of Figiure~\ref{numexamplei}.
\begin{figure}[tbh]
%\centerline{\epsfxsize=9.0in\epsfbox{dynp-9-3-0.64-0.8.eps}}
\centerline{\epsfxsize=7.1in\epsfbox{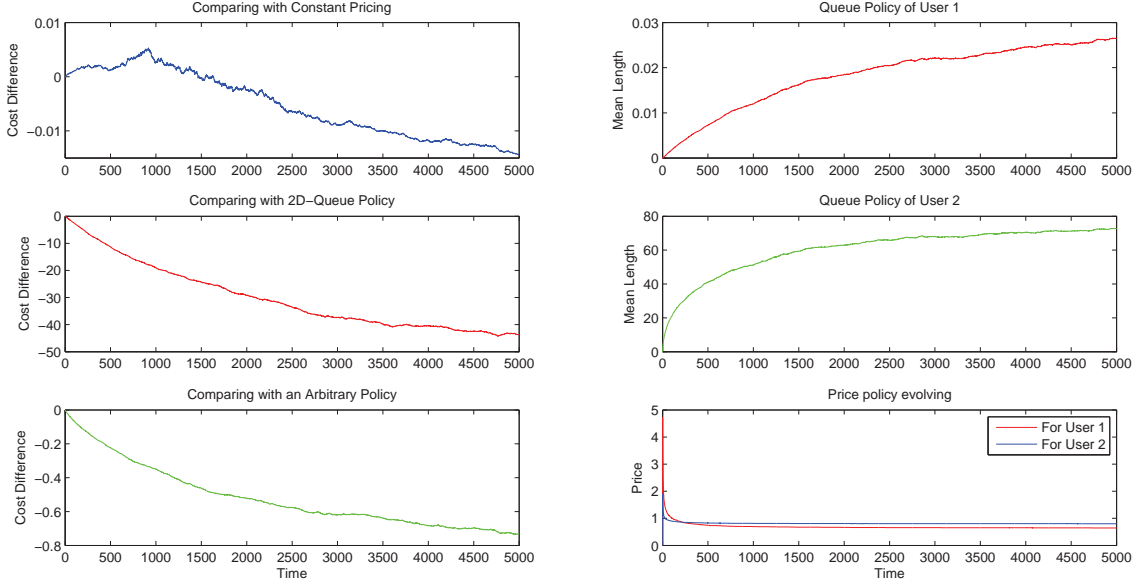}}
\caption{\footnotesize In this simulation, the number of simulation iterative times is $N=6000$, the simulation time interval is $[0,T]$ with $T=200$, which is further divided into $n=5000$ subintervals as explained in Subsection~\ref{singlepoolexample}. Other values of simulation parameters introduced in Definition~\ref{rdrsd} and Subsubsection~\ref{illex} are as follows: initialprice1=9, initialprice2=3, lowerboundprice1=0.64, lowerboundprice2=0.8, $\lambda_{1}=10/3$, $\lambda_{2}=5$, $m_{1}=3$, $m_{2}=1$, $\mu_{1}=1/10$, $\mu_{2}=1/20$, $\alpha_{1}=\sqrt{10}$, $\alpha_{2}=\sqrt{20}$, $\beta_{1}=\sqrt{10}$, $\beta_{2}=\sqrt{20}$, $\zeta_{1}=1$, $\zeta_{2}=\sqrt{2}$, $\rho_{1}=\rho_{2}=1000$, $c_{1}^{2}=c_{2}^{1}=1500$, $\theta_{1}=-1$, $\theta_{2}=-1.2$.}
\label{simexamplei}
\end{figure}

Next, by Theorem~\ref{qdiff}, %and the similar explanations in Dai~\cite{dai:plamod,dai:quacom},
we know that the coefficients of the 1-dimensional RDRS under our dynamic pricing and game-based scheduling policy for the physical workload process $\hat{W}$ can be denoted by
\begin{eqnarray}
&&\left\{\begin{array}{ll}
\hat{b}&=\;\;\theta_{1}/\mu_{1}+\theta_{2}/\mu_{2},\;\;\hat{\sigma}^{E}=\hat{\sigma}^{S}=\left(1/\mu_{1},1/\mu_{2}\right),\;\;\hat{R}=1,\\
%\elabel{ntwocIV}\\
\hat{\sigma}&=\;\;\sqrt{\left(\sum_{j=1}^{2}\hat{\sigma}_{j}^{E}\sqrt{\Gamma_{jj}^{E}}\right)^{2}
+\left(\sum_{j=1}^{2}\hat{\sigma}_{j}^{S}\sqrt{\Gamma_{jj}^{S}}\right)^{2}}.
%\elabel{hatsigma}
\end{array}
\right.
\elabel{ntwocIV}
\end{eqnarray}
Then, based on $\hat{W}$, we can get the dynamic queueing policy by \eq{explictequilI} and its associated dynamic pricing policy through \eq{dp1p2}:
\begin{eqnarray}
&&\hat{Q}(t)=q^{*}(\hat{P}(t),\hat{W}(t)),\;\;\;\hat{P}(t)=\hat{P}(\hat{Q}_{1}(t),\hat{W}(t)).
\elabel{twodynqdynp}
\end{eqnarray}
After determining the initial prices $\hat{P}(0)=(initialprice1, initialprice2)$, we suppose that $\hat{P}(t)$ has the lower bound price protection functionality, i.e., $\hat{P}(t)\in[lowerboundprice1,\infty)$ $\times$ $[lowerboundprice2,\infty)$. Corresponding to \eq{dp1p2}, this truncated price process still owns the Lipschitz continuity as imposed in \eq{pricefun}. Then, by combining the policy in \eq{twodynqdynp} with the simulation algorithms for RDRSs %in Dai~\cite{dai:plamod,dai:quacom},
we can illustrate our policy in \eq{twodynqdynp} is cost-effective in comparing with a constant pricing policy, a 2D-Queue policy and an arbitrarily selected dynamic pricing policy. These simulation comparisons are presented in Figure~\ref{simexamplei}
%Figures~\ref{simexampleI}-\ref{simexampleIII}
with different parameters. The number $N$ of simulation iterative times for these comparisons is 6000 and the simulation time interval is $[0,T]$ with $T=200$, which is further divided into $n=5000$ subintervals. The first graph on the left-column in Figure~\ref{simexamplei}
%each of Figures~\ref{simexampleI}-\ref{simexampleIII}
is the mean total cost difference (MTCD) at each time point $t_{i}$ with $i\in\{0,1,...,5000\}$ between our current dynamic pricing policy in \eq{twodynqdynp} and the constant pricing policy with $P_{1}(t)=P_{2}(t)=1$ for all $t\in[0,\infty)$, i.e.,
\begin{eqnarray}
&&\mbox{MTCD}(t_{i})=\frac{1}{N}\sum_{j=1}^{N}\left(C_{0}(\hat{P}(\omega_{j},t_{i})\hat{Q}(\omega_{j},t_{i}),\rho)
-C_{0}(P(\omega_{j},t_{i})Q(\omega_{j},t_{i}),\rho))\right),
\elabel{costdif}
\end{eqnarray}
where, $\omega_{j}$ denotes the $j$th sample path and the $Q(\omega_{j},t_{i})$ in \eq{costdif} is the queue length corresponding to the constant pricing policy at each time point $t_{i}$. The second graph on the left-column in Figure~\ref{simexamplei}
%each of Figures~\ref{simexampleI}-\ref{simexampleIII}
is the MTCD between our newly designed dynamic pricing policy in \eq{twodynqdynp} and a 2D-Queue policy used as an alternative comparison policy in Dai~\cite{dai:plamod}. For this 2D-Queue policy, the constant pricing with $P_{1}(t)=P_{2}(t)=1$ is employed and the associated $Q(t)$ is presented as a two-dimensional RDRS model as in Dai~\cite{dai:plamod}. The third graph on the left-column in Figure~\ref{simexamplei}
%each of Figures~\ref{simexampleI}-\ref{simexampleIII}
is the MTCD between our current dynamic pricing policy in \eq{twodynqdynp} and an arbitrarily selected dynamic pricing policy given by
\begin{eqnarray}
&&\left\{\begin{array}{ll}
P_{1}(t)&=\;\;\mbox{lowerboundprice1}+\frac{20}{0.05+\sqrt{Q_{1}(t)}},\\
P_{2}(t)&=\;\;\mbox{lowerboundprice2}+\frac{30}{0.1+\sqrt{Q_{2}(t)}}
\end{array}
\right.
\elabel{arbip}
\end{eqnarray}
with the associated queue policy $Q(t)=\hat{Q}(t)$. The first and second graphs on the right-column in Figure~\ref{simexamplei}
%each of Figures~\ref{simexampleI}-\ref{simexampleIII}
display the dynamics of $\hat{Q}(t)$ for both users. The third graph on the right-column in Figure~\ref{simexamplei}
%each of Figures~\ref{simexampleI}-\ref{simexampleIII}
shows the price evolutions corresponding to two users. From the first graph in Figure~\ref{simexamplei}, we can see that the cost is relatively large if the initial prices are relatively high. All of the other comparisons in Figure~\ref{simexamplei}
%Figures~\ref{simexampleI}-\ref{simexampleIII}
show the cost-effectiveness of our policy in \eq{twodynqdynp}.

\subsubsection{The Simulation for Example~\ref{example2}}

Based on the three dual-cost functions in \eq{numcosto}, we can first select any two of the three users for service by formulating the following minimal dual-cost zero-sum game problem for a price parameter $p\in R_{+}^{3}$, a constant $w>0$, and a fixed $c\in{\cal R}$,
\begin{eqnarray}
&&\left\{\begin{array}{ll}
\min_{q\in R_{+}^{3}}C_{0}(pq,c),&\\
%\elabel{adafc}\\
\min_{q\in R_{+}^{3}}C_{0j}(pq,c)&\mbox{subject to}\;\;\;\;q_{1}/\mu_{1}+q_{2}/\mu_{2}\geq w,\\
%\elabel{adafcI}\\
\min_{q\in R_{+}^{3}}\left(-C_{0j_{1}}(pq,c)\right)&\mbox{subject to}\;\;\;\;q_{1}/\mu_{1}+q_{3}/\mu_{3}\geq w,\\
%\elabel{adafcII}\\
\min_{q\in R_{+}^{3}}\left(-C_{0j_{2}}(pq,c)\right)&\mbox{subject to}\;\;\;\;q_{2}/\mu_{2}+q_{3}/\mu_{3}\geq w,
%\elabel{adafcIII}
\end{array}
\right.
\elabel{adafc}
\end{eqnarray}
where, $j\in\{1,2,3\}$, $j_{1}\in\{1,2,3\}\setminus\{j\}$, and $j_{2}\in\{1,2,3\}\setminus\{j,j_{1}\}$, and
\begin{eqnarray}
&&\left\{\begin{array}{ll}
C_{0}(pq,c)&=\;\;C_{1}(p_{1}q_{1},c_{1})+C_{2}(p_{2}q_{2},c_{2})+C_{3}(p_{3}q_{3},c_{3}),\\
%\elabel{3totalct}\\
C_{01}(pq,c)&=\;\;C_{1}(p_{1}q_{1},c_{1})+C_{2}(p_{2}q_{2},c_{2}),\\
%\elabel{312totalct}\\
C_{02}(pq,c)&=\;\;C_{1}(p_{1}q_{1},c_{1})+C_{3}(p_{3}q_{3},c_{3}),\\
%\elabel{313totalct}\\
C_{03}(pq,c)&=\;\;C_{2}(p_{2}q_{2},c_{2})+C_{3}(p_{3}q_{3},c_{3}).
%\elabel{323totalct}
\end{array}
\right.
\elabel{3totalct}
\end{eqnarray}
In other words, if $q^{*}=(q_{1}^{*},q_{2}^{*},q_{3}^{*})$ is a solution to the game problem in \eq{adafc}, and if
\begin{eqnarray}
&&q^{*}_{-j}=\left\{\begin{array}{ll}
(q_{j},q^{*}_{j_{1}},q^{*}_{j_{2}})&\mbox{if}\;\;j=1,\\
(q^{*}_{j_{1}},q_{j},q^{*}_{j_{2}})&\mbox{if}\;\;j=2,\\
(q^{*}_{j_{1}},q^{*}_{j_{2}},q_{j})&\mbox{if}\;\;j=3,
\end{array}
\right.
\elabel{qformpn}
\end{eqnarray}
then, for any two fixed $p,c\in R_{+}^{3}$, we have that
\begin{eqnarray}
&&\left\{\begin{array}{ll}
C_{0}(pq^{*},c)&\leq\;\;C_{0}(pq,c),\\
%\elabel{cpolicyI}\\
C_{0j}(pq^{*},c)&\leq\;\;C_{0j}(pq^{*}_{-j},c),\\
%\elabel{cpolicyII}\\
-C_{0j_{1}}(pq^{*},c)&\leq\;\;-C_{0j_{1}}(pq^{*}_{-j_{1}},c),\\
%\elabel{cpolicyIII}\\
-C_{0j_{2}}(pq^{*},c)&\leq\;\;-C_{0j_{2}}(pq^{*}_{-j_{2}},c).
%\elabel{cpolicyIV}
\end{array}
\right.
\elabel{cpolicyI}
\end{eqnarray}
Furthermore, when two users corresponding to the summation $C_{0j}=C_{k}+C_{l}$ for an index $j\in\{1,2,3\}$ with two associated indices $k,l\in\{1,2,3\}$ as in one of \eq{adafc} are selected, we can propose a 2-stage pricing and queueing policy at each time point by a Pareto minimal dual cost Nash equilibrium point to the non-zero-sum game problem for two fixed $p,c\in R_{+}^{3}$,
\begin{eqnarray}
&&\min_{q\in R_{+}^{3}}C_{0j}(pq,c),\;\;\min_{q\in R_{+}^{3}}C_{k}(pq,c),\;\;\min_{q\in R_{+}^{3}}C_{l}(pq,c).
\elabel{aadafn}
\end{eqnarray}
To wit, if $q^{*}=(q_{k}^{*},q_{l}^{*})$ is a solution to the game problem corresponding to the two users, we have that
\begin{eqnarray}
&&\left\{\begin{array}{ll}
C_{0j}(pq^{*},c)&\geq\;\;C_{0j}(pq,c),\\
%\elabel{acpolicyIn}\\
C_{k}(pq^{*},c)&\geq\;\;C_{1}(pq^{*}_{-k},c)\;\;\;\mbox{with}\;\;\;q^{*}_{-k}=(q_{k},q^{*}_{l}),\\
%\elabel{acpolicyIIn}\\
C_{l}(pq^{*},c)&\geq\;\;C_{l}(pq^{*}_{-l},c)\;\;\;\;\mbox{with}\;\;\;q^{*}_{-l}=(q^{*}_{k},q_{l}).
%\elabel{acpolicyIIIn}
\end{array}
\right.
\elabel{acpolicyIn}
\end{eqnarray}
Thus, for the price parameter $p\in R_{+}^{3}$ and each $w\geq 0$, it follows from \eq{adafc}-\eq{cpolicyI} and \eq{aadafn}-\eq{acpolicyIn} that our queueing policy $(q_{1}^{*}(p,w),q_{2}^{*}(p,w),q_{3}^{*}(p,w))$ can be designed by
\begin{eqnarray}
&&\;\;\;\;\left\{\begin{array}{ll}
\left\{\begin{array}{ll}
q_{1}^{*}(p,w)=\bar{g}_{1}(p_{1},p_{2},w),\\
q_{2}^{*}(p,w)=\bar{g}_{2}(p_{1},p_{2},q_{1}^{*},w)
\end{array}
\right.
&\mbox{if}\;\;
C_{01}(pq^{*},c)\leq\min\left\{C_{02}(pq^{*},c),C_{03}(pq^{*},c)\right\},\\
\left\{\begin{array}{ll}
q_{1}^{*}(p,w)=\hat{g}_{1}(p_{1},p_{3},w),\\
q_{3}^{*}(p,w)=\hat{g}_{3}(p_{1},p_{3},q_{1}^{*},w)
\end{array}
\right.
&\mbox{if}\;\;
C_{02}(pq^{*},c)\leq\min\left\{C_{01}(pq^{*},c),C_{03}(pq^{*},c)\right\},\\
\left\{\begin{array}{ll}
q_{2}^{*}(p,w)=\bar{g}_{2}(p_{3},p_{2},q_{3}^{*},w),\\
q_{3}^{*}(p,w)=\bar{g}_{1}(p_{3},p_{2},w)
\end{array}
\right.
&\mbox{if}\;\;
C_{03}(pq^{*},c)\leq\min\left\{C_{01}(pq^{*},c),C_{02}(pq^{*},c)\right\},
\end{array}
\right.
\elabel{qacpolicyIn}
\end{eqnarray}
where, the function $\bar{g}_{1}$ is given in \eq{explictequilI} and $\hat{g}_{1}$ is calculated in a similar way as follows,
\begin{eqnarray}
&&\;\;\;\;\left\{\begin{array}{ll}
\hat{g}_{1}(p_{1},p_{3},w)&=\frac{(p_{3}^{2}\mu_{3}w)/(\mu_{1}c_{3})}{(p_{1}^{2}/\mu_{1}c_{1})+(p_{3}^{2}\mu_{3}/\mu_{1}^{2}c_{3})},\\
\hat{g}_{3}(p_{1},p_{3},q_{1},w)&=\;\;\mu_{3}\left(w-(q_{1}/\mu_{1})\right).
\end{array}
\right.
\elabel{hatgq1q3}
\end{eqnarray}
The intersection point of $\hat{g}_{1}$ and $\hat{g}_{3}$ in terms of $q_{1}$ is a Pareto optimal Nash equilibrium point as shown in the lower-left graph of Figure~\ref{numexamplei}. Furthermore, based on \eq{qacpolicyIn}-\eq{hatgq1q3}, we can inversely determine our pricing policy $p=(p_{1},p_{2},p_{3})$ as follows,
\begin{eqnarray}
&&\;\;\;\;\left\{\begin{array}{ll}
\left\{\begin{array}{ll}
p_{1}(q^{*}_{1},w)=\kappa^{2}(q_{1}^{*},w),\\
p_{2}(q^{*}_{1},w)=\kappa(q^{*}_{1},w)
\end{array}
\right.
&\mbox{if}\;\;
C_{01}(pq^{*},c)\leq\min\left\{C_{02}(pq^{*},c),C_{03}(pq^{*},c)\right\},\\
\left\{\begin{array}{ll}
p_{1}(q^{*}_{1},w)=\varpi(q^{*}_{1})\hat{\kappa}(q^{*}_{1},w),\\
p_{3}(q^{*}_{1},w)=\varpi(q^{*}_{1})\sqrt{\hat{\kappa}(q_{1}^{*},w)}
\end{array}
\right.
&\mbox{if}\;\;
C_{02}(pq^{*},c)\leq\min\left\{C_{01}(pq^{*},c),C_{03}(pq^{*},c)\right\},\\
\left\{\begin{array}{ll}
p_{2}(q_{3}^{*},w)=\kappa(q_{3}^{*},w),\\
p_{3}(q_{3}^{*},w)=\kappa^{2}(q^{*}_{3},w)
\end{array}
\right.
&\mbox{if}\;\;
C_{03}(pq^{*},c)\leq\min\left\{C_{01}(pq^{*},c),C_{02}(pq^{*},c)\right\},
\end{array}
\right.
\elabel{pqacpolicyIn}
\end{eqnarray}
where, $\kappa$ is defined in \eq{pqexplictequilI} and $\hat{\kappa}$ can be calculated in the same way as follows,
\begin{eqnarray}
&&\hat{\kappa}(q^{*}_{1},w)=\frac{\mu_{3}c_{1}}{q^{*}_{1}c_{3}}\left(w-\frac{q^{*}_{1}}{\mu_{1}}\right).
\elabel{hatkappa}
\end{eqnarray}
Furthermore, $\varpi(q^{*}_{1})$ in \eq{pqacpolicyIn} is a nonnegative function in terms of $q^{*}_{1}$ and it is taken to be the unity in the drawing of dynamic pricing evolving in the lower-graph of Figure~\ref{numexamplei} with $w=10000$.

To show the cost-effectiveness of our queueing policy in \eq{qacpolicyIn} with its associated pricing policy in \eq{pqacpolicyIn}, we present an arbitrarily selected stochastic pooling policy for the purpose of comparisons as follows,
\begin{eqnarray}
&&\;\;\left\{\begin{array}{ll}
\left\{\begin{array}{ll}
q_{1}^{*}(p,w)=\bar{g}_{1}(p_{1},p_{2},w),\\
q_{2}^{*}(p,w)=\bar{g}_{2}(p_{1},p_{2},q_{1}^{*},w)
\end{array}
\right.
&\mbox{if}\;\;
u\in\left[0,\frac{1}{3}\right),\\
\left\{\begin{array}{ll}
q_{1}^{*}(p,w)=\hat{g}_{1}(p_{1},p_{3},w),\\
q_{3}^{*}(p,w)=\hat{g}_{3}(p_{1},p_{3},q_{1}^{*},w)
\end{array}
\right.
&\mbox{if}\;\;
u\in\left[\frac{1}{3},\frac{2}{3}\right),\\
\left\{\begin{array}{ll}
q_{2}^{*}(p,w)=\bar{g}_{2}(p_{3},p_{2},q_{3}^{*},w),\\
q_{3}^{*}(p,w)=\bar{g}_{1}(p_{3},p_{2},w)
\end{array}
\right.
&\mbox{if}\;\;
u\in\left[\frac{2}{3},1\right],
\end{array}
\right.
\elabel{rqacpolicyIn}
\end{eqnarray}
where, $u$ is a uniformly distributed random number.

After determining the initial price vector $\hat{P}(0)=(initialprice1,initialprice2,initialprice3)$, we suppose that $\hat{P}(t)$ has the lower bound price protection and the upper bound constraint functionalities, i.e., $\hat{P}(t)\in[lowerboundprice1$, $upperboundprice1)$ $\times$ $[lowerboundprice2$, $upperboundprice2)$ $\times$ $[lowerboundprice3$, $upperboundprice3)$. Corresponding to \eq{pqacpolicyIn}, this truncated price process still own the Lipschitz continuity as imposed in \eq{pricefun}. Then, by the similar explanations used for \eq{costdif}, we can conduct the corresponding simulation comparisons for this example as shown in Figures~\ref{3simexamplei}-\ref{3simexampleii}. The cost value evolution based on our queueing policy in \eq{qacpolicyIn} with its associated pricing policy in \eq{pqacpolicyIn} is shown in the first graph of the left-column in each of Figures~\ref{3simexamplei}-\ref{3simexampleii}. Its MTCD in \eq{costdif} compared with the arbitrarily selected stochastic pooling policy in \eq{rqacpolicyIn} is displayed in the first graph of the right-column in each of Figures~\ref{3simexamplei}-\ref{3simexampleii}. The cost value evolution based on our queueing policy in \eq{qacpolicyIn} with constant pricing (i.e., $p_{1}=p_{2}=p_{3}$) is shown in the second graph of the left-column in each of Figures~\ref{3simexamplei}-\ref{3simexampleii}. In this constant pricing case, its MTCD in \eq{costdif} compared with the arbitrarily selected stochastic pooling policy in \eq{rqacpolicyIn} is displayed in the second graph of the right-column in each of Figures~\ref{3simexamplei}-\ref{3simexampleii}. The MTCD based on our queueing policy in \eq{qacpolicyIn} with its associated pricing policy in \eq{pqacpolicyIn} and with the constant pricing policy is shown in the third graph of the left-column in each of Figures~\ref{3simexamplei}-\ref{3simexampleii}. The price evolutions for the three users are shown in the third graph of the right-column in each of Figures~\ref{3simexamplei}-\ref{3simexampleii}. In the special case with parameters as shown in Figure~\ref{3simexampleii}, the three price evolutions are the same. Furthermore, the MTCD between our dynamic pricing policy in \eq{pqacpolicyIn} and the constant pricing policy is the number 0 as shown in the third graph of the left-column in Figure~\ref{3simexampleii}.

\section{Justification of RDRS modeling}\label{rdrsm}

In this section, we theoretically prove the correctness of our RDRS modeling presented in Theorem~\ref{qdiff}.

%\subsection{Model Justification Part I: under A Game Scheduling
%Policy}\label{apolicy}

\subsection{The required conditions}

In this subsection, we present the required conditions and assumptions in proving our RDRS modeling. The utility functions can be either simply taken as the well-known proportionally fair and minimal potential delay allocations as used in \eq{numutility} for Example~\ref{singlepoolexample} or generally taken such that the existence of a mixed saddle point and Pareto maximal-utility Nash equilibrium policy to the game problem in \eq{gameoptoo}-\eq{gameoptoIV} is guaranteed. More precisely, for each given $p\in R_{+}^{J}$, we can assume that $U_{vj}(p_{j}q_{j},c_{vj})$ for each $j\in{\cal J}(v)$ and $v\in{\cal V}(j)$ is defined on $R_{+}^{J}$. It is second-order differentiable and satisfies
\begin{eqnarray}
&&\;\;\left\{\begin{array}{ll}
U_{vj}(0,c_{vj})=0,\\
%\elabel{uconI}\\
U_{vj}(p_{j}q_{j},c_{vj})=\Phi_{vj}(p_{j}q_{j})\Psi_{v}(c_{vj})\;\mbox{is strictly increasing/concave in}\;c_{vj}\;\mbox{for}\;p_{j}q_{j}>0,\\
%\elabel{uconII}\\
\Psi_{v}(\nu_{j}c_{vj})=\Psi_{v}(\nu_{j})\Psi_{v}(c_{vj})\;\;\mbox{or}\;\;\Psi_{v}(\nu_{j}c_{vj})=\Psi_{v}(\nu_{j})+\Psi_{v}(c_{vj})\;\;\mbox{for constant}\;\;\nu_{j}\geq 0,\\
%\elabel{faddedcon}\\
\frac{\partial U_{vj}(p_{j}q_{j},c_{vj})}{\partial c_{vj}}\;\;\mbox{is strictly increasing in}\;\;p_{j}q_{j}\geq 0,\;\;\\
%\elabel{uconIII}\\
\frac{\partial U_{vj}(0,c_{vj})}{\partial c_{vj}}=0\;\;\mbox{and}\;\;\lim_{q_{j}\rightarrow\infty}\frac{\partial U_{vj}(p_{j}q_{j},c_{vj})}{\partial c_{vj}}=+\infty\;\;\mbox{for each}\;\;c_{vj}>0.
%\elabel{uconIV}
\end{array}
\right.
\elabel{uconI}
\end{eqnarray}
Furthermore, we suppose that $\{U_{vj}(p_{j}q_{j},c_{vj}),j\in{\cal J}(v),v\in{\cal V}(j)\}$ satisfies the radial homogeneity condition at each given time point $t\in[0,\infty)$. In other words, for any scalar $a>0$, each $q>0$, $i\in{\cal K}$, $v\in{\cal V}$, and each $j_{l}\in{\cal M}(i,v,t)$ with $l\in\{1,...,M_{v}\}$, its Pareto maximal utility Nash equilibrium point for the game has the radial homogeneity
\begin{eqnarray}
&&c_{vj_{l}}(apq,i)=c_{vj_{l}}(pq,i).
\elabel{homcon}
\end{eqnarray}
In addition, we introduce a sequence of independent Markov processes indexed by $r\in{\cal R}$, i.e., $\{\alpha^{r}(\cdot),r\in{\cal R}\}$. These systems all have the same basic structure as presented in the last section except the arrival rates $\lambda^{r}_{j_{l}}(i)$ and the holding time rates $\gamma^{r}(i)$ for all $i\in{\cal K}$, which may vary with $r\in{\cal R}$. Here, we suppose that they satisfy the heavy traffic condition
\begin{eqnarray}
&& r\left(\lambda_{j_{l}}^{r}(i)-\lambda_{j_{l}}(i)\right)m_{j_{l}}(i)
\rightarrow\theta_{j_{l}}(i)\;\;\mbox{as}\;\;r\rightarrow\infty,
\;\;\gamma^{r}(i)=\frac{\gamma(i)}{r^{2}},
\elabel{heavytrafficc}
\end{eqnarray}
where, $\theta_{j_{l}}(i)\in R$ is some constant for each $i\in{\cal K}$.
%Note that, $\theta_{j_{l}}(i)\in R$ is some constant for each $i\in{\cal K}$,
%which can be chosen optimally in certain environment (see, e.g., Dai and
%Jiang~\cite{daijia:stoopt}).
Moreover, we suppose that the nominal arrival rate $\lambda_{j_{l}}(i)$ is given by
\begin{eqnarray}
&&\lambda_{j_{l}}(i)m_{j_{l}}(i)\equiv\mu_{j_{l}}\rho_{j_{l}}(i)
\elabel{brhocon}
\end{eqnarray}
%In practice, a nominal arrival rate is corresponding to the number of links
%allowed in a data service system, which can be realized by the technique
%of admission control (see, e.g., Dai~\cite{dai:optcon} and references therein).
%Furthermore, $\rho_{j_{l}}(i)$ $j\in{\cal J}(v)$ {\textcolor{red}{and all $v\in{\cal V}(j)$}}
and $\rho_{j_{l}}(i)$ in \eq{brhocon} for $j_{l}\in{\cal M}(i,v,t)$ with
$l\in\{1,...,M_{v}\}$ is the nominal throughput determined by
\begin{eqnarray}
&&\rho_{j_{l}}(i)=\sum_{v\in{\cal V}(j_{l})}\rho_{vj_{l}}(i)\;\;\;
\mbox{and}\;\;\;\rho_{vj_{l}}(i)=\nu_{vj_{l}}\bar{\rho}_{vj_{l}}(i)
\elabel{brhoconI}
\end{eqnarray}
with $\rho_{v\cdot}(i)\in{\cal O}_{v}(i)$ that is corresponding to the dimension $M_{v}$. In addition, $\nu_{v\cdot}$ and $\bar{\rho}_{v\cdot}(i)$ are an $J_{v}$-dimensional constant vector and a reference service rate vector, respectively, at service pool $v$, satisfying
\begin{eqnarray}
\sum_{j_{l}\in{\cal M}(i,v,t)\bigcap{\cal J}(v)}\nu_{j_{l}}&=&J_{v},\;\nu_{j_{l}}\geq 0\;\;\;\mbox{are constants for all}\;\;j_{l}\in{\cal M}(i,v,t)\cap{\cal J}(v),
\elabel{weightave}\\
\;\;\;\;\;\;\;\;\;\;
\sum_{j_{l}\in{\cal M}(i,v,t)\cap{\cal J}(v)}\bar{\rho}_{vj_{l}}(i)&=&{\cal C}_{U_{v}}(i)\;\;\mbox{and}\;\;\bar{\rho}_{vj_{1}}(i)=\bar{\rho}_{vj_{l}}(i)\;\;\mbox{for all}\;\;j_{l}\in{\cal M}(i,v,t)\cap{\cal J}(v).
\elabel{rhoji}
\end{eqnarray}
\begin{remark}
By \eq{hmiddle}, $\bar{\rho}_{v\cdot}(i)$ for each $i\in{\cal K}$ and $v\in{\cal V}(j_{l})$ can indeed be selected, which satisfy the second condition in \eq{rhoji}. Thus, the nominal throughput $\rho(i)$ in \eq{brhocon} can be determined. One simple example that satisfies these conditions is to take $\nu_{vj_{l}}=1$ for all $j_{l}\in{\cal M}(i,v,t)\cap{\cal J}(v)$ and $v\in{\cal V}(j_{l})$. Thus, the conditions in \eq{brhocon}-\eq{rhoji} mean that the system manager wishes to maximally and fairly allocate capacity to all users. Moreover, the design parameters $\lambda_{j_{l}}(i)$ for all $j_{l}\in{\cal M}(i,v,t)\cap{\cal J}$ and each $i\in{\cal K}$ can be determined by \eq{brhocon}.
\end{remark}

Next, we assume that the inter-arrival time associated with the $k$th arriving job batch to the system indexed by $r\in{\cal R}$ is given by
\begin{eqnarray}
&&u_{j_{l}}^{r}(k,i)=\frac{\hat{u}_{j_{l}}(k)}{\lambda_{j_{l}}^{r}(i)}\;\;\mbox{for each}\;\;j_{l}\in{\cal M}(i,v,t)\cap{\cal J},\;k\in\{1,2,...\},
\;i\in{\cal K},
\elabel{intercon}
\end{eqnarray}
where, the $\hat{u}_{j_{l}}(k)$ does not depend on $r$ and $i$. Moreover, it has mean one and finite squared coefficient of variation $\alpha_{j_{l}}^{2}$. In addition, the number of packets, $w_{j_{l}}(k)$, and the packet length $v_{j_{l}}(k)$ are assumed not to change with $r$. Thus, it follows from the heavy traffic condition in \eq{heavytrafficc} for the $r$th environmental state process $\alpha^{r}(\cdot)$ with $r\in{\cal R}$ that $\alpha^{r}(r^{2}\cdot)$ and $\alpha(\cdot)$ equal to each other in distribution since they own the same generator matrix (see, e.g., the definition in pages 384-388 of Resnick~\cite{res:advsto}). Therefore, under the sense of distribution, all of the systems indexed by $r\in{\cal R}$ in \eq{rsqueue} has the same random environment over any time interval $[0,t]$.

\subsection{Proof of Theorem~\ref{qdiff}}\label{proofmain}

%For clearance, we divide our proof into several parts as given
%in the subsequent subsections.

%\subsection{Diffusion and Fluid Scaled Processes}\label{diffluid}

First, it follows from the second condition in \eq{heavytrafficc} that the processes $\alpha^{r}(r^{2}\cdot)$ for each $r\in{\cal R}$ and $\alpha(\cdot)$ are equal in distribution. Hence, without loss of generality, we can assume that
\begin{eqnarray}
&&\alpha^{r}(r^{2}t)=\alpha(t)\;\;\;\mbox{for each}\;\;\;r\in{\cal R}\;\;\;\mbox{and}\;\;\;t\in[0,\infty).
\elabel{alphaequal}
\end{eqnarray}
Thus, for each $j\in{\cal J}$, $r\in{\cal R}$ and by the radial homogeneity of $\Lambda(pq,i)$ of the policy in \eq{homcon}, we can define the fluid and diffusion scaled processes as follows,
\begin{eqnarray}
E^{r}_{j}(\cdot)&\equiv&A^{r}_{j}(r^{2}\cdot),
\elabel{enee}\\
\bar{T}^{r}_{j}(\cdot)&\equiv&\int_{0}^{\cdot}\Lambda_{j}\left(\bar{P}^{r}(s)\bar{Q}^{r}(s),\alpha(s),s\right)ds=\frac{1}{r^{2}}T_{j}^{r}(r^{2}\cdot),
\elabel{expresslambda}\\
\bar{Q}_{j}^{r}(t)&\equiv&\frac{1}{r^{2}}Q^{r}_{j}(r^{2}t),
\elabel{barqueue}\\
\bar{P}^{r}_{j}(t)&=&f_{j}(\bar{Q}^{r}_{j}(t),\alpha(t)),
\elabel{barpricefun}\\
%\bar{W}^{V,r}(t)&\equiv&\frac{1}{r^{2}}W^{V,r}(r^{2}t),
%\elabel{qebar}\\
%\bar{Y}^{V,r}(t)&\equiv&\frac{1}{r^{2}}Y^{V,r}(r^{2}t),
%\elabel{qebarI}\\
\bar{E}^{r}_{j}(t)&\equiv&\frac{1}{r^{2}}E_{j}^{r}(t),
\elabel{qebarII}\\
\bar{S}^{r}_{j}(t)&\equiv&\frac{1}{r^{2}}S_{j}^{r}(r^{2}t).
\elabel{qebarIII}
\end{eqnarray}
Then, it follows from \eq{queuelength}, \eq{alphaequal}, the assumptions among arrival and service processes that
\begin{eqnarray}
&&\hat{Q}^{r}_{j}(\cdot)=\frac{1}{r}E^{r}_{j}(\cdot)-\frac{1}{r}S^{r}_{j}(\bar{T}^{r}_{j}(\cdot)).
\elabel{centerhatQ}
\end{eqnarray}
Furthermore, for each $j\in{\cal J}$, let
\begin{eqnarray}
&&\hat{E}^{r}(\cdot)=(\hat{E}^{r}_{1}(\cdot),...,\hat{E}^{r}_{J}(\cdot))'\;\;\mbox{with}\;\;\hat{E}^{r}_{j}(\cdot)=\frac{1}{r}
\left(A^{r}_{j}(r^{2}\cdot)-r^{2}\bar{\lambda}^{r}_{j}(\cdot)\right),
\elabel{ecenterrs}\\
&&\hat{S}^{r}(\cdot)=(\hat{S}^{r}_{1}(\cdot),...,\hat{S}^{r}_{J}(\cdot))'\;\;\;\;\mbox{with}\;\;
\hat{S}_{j}^{r}(\cdot)=\frac{1}{r}\left(S_{j}(r^{2}\cdot)-\mu_{j}r^{2}\cdot\right),
\elabel{scenterrs}
\end{eqnarray}
where,
\begin{eqnarray}
&&\bar{\lambda}^{r}_{j}(\cdot)\equiv\int_{0}^{\cdot}m_{j}(\alpha(s),s)\lambda_{j}^{r}(\alpha(s),s)ds
=\frac{1}{r^{2}}\int_{0}^{r^{2}\cdot}m_{j}(\alpha^{r}(s),r^{2}s)\lambda_{j}^{r}(\alpha^{r}(s),r^{2}s)ds.
\elabel{averagerate}
%&=&\int_{0}^{\cdot}m_{j}(\alpha^{r}(r^{2}s),r^{2}s)\lambda_{j}^{r}(\alpha^{r}(r^{2}s),r^{2}s)ds
%\nonumber\\
%&=&\frac{1}{r^{2}}\int_{0}^{r^{2}\cdot}m_{j}(\alpha^{r}(s),r^{2}s)\lambda_{j}^{r}(\alpha^{r}(s),r^{2}s)ds.
%\nonumber
\end{eqnarray}
For convenience, we define
\begin{eqnarray}
\bar{\lambda}^{r}(\cdot)&=&\left(\bar{\lambda}^{r}_{1}(\cdot),...,\bar{\lambda}^{r}_{J}(\cdot)\right)'.
\elabel{lambdaeo}
\end{eqnarray}
In addition, we let $\bar{Q}^{r}(\cdot)$, $\bar{E}^{r}(\cdot)$, $\bar{S}^{r}(\cdot)$, and $\bar{T}^{r}(\cdot)$ be the associated vector processes. Then, for the processes in \eq{enee}-\eq{centerhatQ}, we define the corresponding fluid limit related processes,
\begin{eqnarray}
\bar{Q}_{j}(t)&=&\bar{Q}_{j}(0)+\bar{\lambda}_{j}(t,\zeta_{t}(\cdot))-\mu_{j}\bar{T}_{j}(t)\;\;\mbox{for each}\;\;j\in{\cal J},
\elabel{limqtnon}
%\bar{W}(t)&=&\sum_{j=1}^{J}\frac{\bar{Q}_{j}(t)}{\mu_{j}}
%=\bar{W}(0)+\bar{Y}(t),\elabel{limwtnon}\\
%\bar{Y}(t)&=&\sum_{j=1}^{J}\left(\int_{0}^{t}\rho_{j}(\alpha(s))ds
%-\bar{T}_{j}(t)\right),\elabel{limwtnonI}\\
\end{eqnarray}
where, $\zeta_{t}(\cdot)$ denotes a process depending on the external environment, i.e.,
\begin{eqnarray}
\bar{\lambda}(t) &=&\left(\bar{\lambda}_{1}(t),...,\bar{\lambda}_{J}(t)\right)',\;\; \bar{\lambda}_{j}(t)\equiv\int_{0}^{t}m_{j}\lambda_{j}(\alpha(s),s)ds.
\elabel{lambdae}
\end{eqnarray}
Furthermore, we have that
\begin{eqnarray}
\bar{T}_{j}(t)&=&\int_{0}^{t}\bar{\Lambda}_{j}(\bar{P}(s)\bar{Q}(s),\alpha(s),s)ds,
\elabel{limtnon}\\
\bar{P}(t)&=&f(\bar{Q}(t),\alpha(t)),
\elabel{barpricefun}
\end{eqnarray}
where, for each $i\in{\cal K}$ and $t\in[0,\infty)$, we have that
\begin{eqnarray}
\bar{\Lambda}_{j}(pq,i,t)&=&\left\{\begin{array}{ll}
\Lambda_{j}(pq,i,t)&\mbox{if}\;\;q_{j}>0,j\in\bigcup_{v\in{\cal V}}{\cal M}(i,v,t),\\
\rho_{j}(i,t)&\mbox{if}\;\;q_{j}>0,j\nsubseteq\bigcup_{v\in{\cal V}}{\cal M}(i,v,t),\\
\rho_{j}(i,t)\;\;\;\;&\mbox{if}\;\;q_{j}=0.
\end{array}\right.
\elabel{barLAMnonI}
\end{eqnarray}
Then, we have the following lemma concerning the weak convergence to a stochastic fluid limit process under our game-competition based dynamic pricing and scheduling strategy.
\begin{lemma}\label{fluidlemma}
Assume that the initial queue length $\bar{Q}^{r}(0)\Rightarrow\bar{Q}(0)$ along $r\in{\cal R}$. Then, the joint convergence in distribution along a
subsequence of ${\cal R}$ is true under our game-competition based dynamic pricing and scheduling strategy in \eq{gameoptoo} and \eq{policyv} with the conditions required by Theorem~\ref{qdiff},
\begin{eqnarray}
&&\left(\bar{E}^{r}(\cdot),\bar{S}^{r}(\cdot),\bar{T}^{r}(\cdot),\bar{Q}^{r}(\cdot)\right)\Rightarrow\left(\bar{E}(\cdot),\bar{S}(\cdot),
\bar{T}(\cdot),\bar{Q}(\cdot)\right).
\elabel{fluidcon}
\end{eqnarray}
In addition, if $\bar{Q}(0)=0$, the convergence is true along the whole ${\cal R}$ and the limit satisfies
\begin{eqnarray}
&&\bar{E}(\cdot)=\bar{\lambda}(\cdot),\;\;\bar{S}(\cdot)=\mu(\cdot),\;\;\bar{T}(\cdot)=\bar{c}(\cdot),\;\;\bar{Q}(\cdot)=0,
\elabel{fluidlimit}
\end{eqnarray}
where, $\bar{\lambda}(\cdot)$ is defined in \eq{lambdae}, $\mu(\cdot)\equiv(\mu_{1},...,\mu_{J})'\cdot$, and $\bar{c}(\cdot)$ is defined by
\begin{eqnarray}
\;\;\;\;\bar{c}(t)&=&\left(\bar{c}_{1}(t),...,\bar{c}_{J}(t)\right)'\;\;\mbox{and}\;\;\bar{c}_{j}(t)\equiv\int_{0}^{t}\rho_{j}(\alpha(s),s)ds\;\;
\mbox{for each}\;\;j\in{\cal J}.
\elabel{nbarvc}
\end{eqnarray}
\end{lemma}
\begin{proof}
%\proof{Proof of Lemma \ref{fluidlemma}.}
First, by the proof of Lemma 1 in Dai~\cite{dai:optrat} and the implicit function theorem, we can show that the pricing function $f$ constructed through \eq{costminpo}, \eq{gameoptoo}, and \eq{policyv} can be assumed to be Lipschitz continuous. Then, by extending the proof of Lemma 3 in Dai~\cite{dai:optrat} and under the conditions in \eq{uconI}-\eq{homcon} and the just illustrated Lipschitz continuity for $f$, we know that, if $\Lambda(pq,i)\in F_{{\cal Q}}(i)$ for each $i\in{\cal K}$ is a given mixed saddle and Pareto optimal Nash equilibrium policy to the game problem in \eq{gameoptoo} and $\{p^{l}q^{l},l\in{\cal R}\}$ is a sequence of valued queue lengths, which satisfies $p^{l}q^{l}\rightarrow pq\in R_{+}^{J}$ as $l\rightarrow\infty$. Then, for each $j\in{\cal J}\setminus {\cal Q}(q)$ and $v\in{\cal V}(j)$, we have that
\begin{eqnarray}
&&\Lambda_{vj}(p^{l}q^{l},i)\rightarrow\Lambda_{vj}(pq,i)\;\;\mbox{as}\;\;l\rightarrow\infty.
\elabel{ralloccon}
\end{eqnarray}

Second, due to the proof of Lemma 7 in Dai~\cite{dai:optrat}, we only need to prove that a weak fluid limit on the right-hand side of \eq{fluidcon} satisfies \eq{nbarvc}. In doing so, we suppose that the weak fluid limit on the right-hand side of \eq{fluidcon} corresponds to a subsequence of the right-hand side of \eq{fluidcon}, which is indexed by $r_{l}\in{\cal R}$ with $l\in\{1,2,...\}$. Furthermore, it follows from \eq{expresslambda}, \eq{tjqalpha}, and the discussion in the proof of Lemma 7 of Dai~\cite{dai:optrat} that the fluid limit process on the right-hand side of \eq{fluidcon} is uniformly Lipschitz continuous a.s. Thus, our discussion can focus on a fixed sample path and each regular point $t>0$ over an interval $(\tau_{n-1},\tau_{n})$ with $n\in\{1,2,...\}$ for $\bar{T}_{j}$ with $j\in{\cal J}$. More precisely, it follows from \eq{limqtnon} that $\bar{Q}$ is differential at $t$ and satisfies
\begin{eqnarray}
&&\frac{d\bar{Q}_{j}(t)}{dt}=m_{j}\lambda_{j}(\alpha(t),t)-\mu_{j}\frac{d\bar{T}_{j}(t)}{dt}
\elabel{barqbart}
\end{eqnarray}
for each $j\in{\cal J}$. If $\bar{Q}_{j}(t)=0$ for some $j\in{\cal J}$, then it follows from $\bar{Q}_{j}(\cdot)\geq 0$ that
\begin{eqnarray}
&&\frac{d\bar{Q}_{j}(t)}{dt}=0\;\;\mbox{which implies that}\;\;\frac{d\bar{T}_{j}(t)}{dt}=\frac{m_{j}\lambda_{j}(\alpha(t),t)}{\mu_{j}} =\rho_{j}(\alpha(t),t).
\elabel{firstfc}
\end{eqnarray}
If $\bar{Q}_{j}(t)>0$ for the $j\in{\cal J}$, there is a finite interval $(a,b)\in[0,\infty)$ containing $t$ in it such that $\bar{Q}_{j}(s)>0$ for all $s\in(a,b)$ and hence we can take sufficiently small $\delta>0$ such that $\bar{Q}_{j}(t+s)>0$ with $s\in(0,\delta)$. Furthermore, by \eq{pricefun}, $P_{j}(t+s)>0$. Now, let $r_{l}$ with $l\in{\cal R}$ be the subsequence ${\cal R}$ and let $\delta_{l}\in(0,\delta]$ be a sequence such that $\delta_{l}\rightarrow 0$ as $l\rightarrow\infty$ while $\Lambda_{j}$ determined by a same group of users over $(0,\delta_{l}]$. Then, it follows from \eq{expresslambda} that
\begin{eqnarray}
&&\left|\frac{1}{\delta_{l}}\left(\bar{T}^{r_{l}}_{j}(t+\delta_{l})-\bar{T}^{r_{l}}_{j}(t)\right)-\Lambda_{j}(\bar{P}(t)\bar{Q}(t),\alpha(t),t)\right|
\elabel{bartevl}\\
&\leq&\frac{1}{\delta_{l}}\int_{0}^{\delta_{l}}\Big|\Lambda_{j}(\bar{P}^{r_{l}}(t+s)\bar{Q}^{r_{l}}(t+s),\alpha(t+s),t+s)
\nonumber\\
&&\;\;\;\;\;\;\;\;\;\;\;\;\;\;\;\;\;\;\;\;-\Lambda_{j}(\bar{P}(t+s)\bar{Q}(t+s),\alpha(t+s),t+s)\Big|ds
\nonumber\\
&&+\frac{1}{\delta_{l}}\int_{0}^{\delta_{l}}\left|\Lambda_{j}(\bar{P}(t+s)\bar{Q}(t+s),\alpha(t+s),t+s)
-\Lambda_{j}(\bar{P}(t)\bar{Q}(t),\alpha(t),t)\right|ds
\nonumber\\
\nonumber\\
&\rightarrow&0\;\;\;\;\mbox{as}\;\;l\rightarrow\infty,
\nonumber
\end{eqnarray}
where, the last claim in \eq{bartevl} follows from the Lebesgue dominated convergence theorem, the right-continuity of $\alpha(\cdot)$, the Lipschitz continuity of $\bar{Q}(\cdot)$, and the fact in \eq{ralloccon}. Since $t$ is a regular point of $\bar{T}$, it follows from \eq{bartevl} that
\begin{eqnarray}
&&\frac{d\bar{T}_{j}(t)}{dt} =\frac{d\bar{T}_{j}(t^{+})}{dt}=\bar{\Lambda}_{j}(\bar{Q}(t),\alpha(t),t)\;\;\mbox{for each}\;\;j\in{\cal J} \elabel{fracbart}
\end{eqnarray}
which implies that the claims in \eq{limtnon}-\eq{barLAMnonI} are true.

Along the line of the proofs for Lemma 4.2 in Dai~\cite{dai:plamod}, Lemma 4.1 in Dai~\cite{dai:quacom}, and Lemma 7 in Dai~\cite{dai:optrat}, it suffices to prove the claim that $\bar{Q}(\cdot)=0$ in \eq{fluidlimit} holds for the purpose of our current paper. In fact, for each $i\in{\cal K}$ and $l\in\{1,...,M_{v}\}$, we define
\begin{eqnarray}
&&\psi(pq,i)\equiv\sum_{v\in{\cal V}}\psi_{v}(pq,i)=\sum_{v\in{\cal V}}\sum_{j_{l}\in{\cal M}(i,v)\bigcap{\cal J}(v)}C_{vj_{l}}(p_{j_{l}}q_{j_{l}},\rho_{vj_{l}}(i)).
\elabel{psicj}
\end{eqnarray}
Then, at each regular time $t\geq 0$ of $\bar{Q}(t)$ over time interval $(\tau_{n-1},\tau_{n})$ with a given $n\in\{1,2,...\}$, we have that
\begin{eqnarray}
&&\;\;\;\;\;\frac{d\psi(\bar{P}(t)\bar{Q}(t),\alpha(t))}{dt}
\elabel{lyapunov}\\
%&=&\sum_{v\in{\cal V}}\sum_{j_{l}\in{\cal M}(i,v,t)\bigcap{\cal J}(v)}\left(\frac{d(\bar{P}_{j_{l}}(t)\bar{Q}_{j_{l}}(t))}{dt}\frac{\partial %C_{vj_{l}}(\bar{P}_{j_{l}}(t)\bar{Q}_{j_{l}}(t),\rho_{vj_{l}}(\alpha(t),t))}{\partial(\bar{P}_{j_{l}}(t)\bar{Q}_{j_{l}}(t))}\right.
%\nonumber\\
%&&\;\;\;\;\;\;\;\;\;\;\;\;\;\;\;\;\;\;\;\;\;\;\;\;\;\;\;\;\;\;\;\;\;\;\;\;
%\left.+\frac{d\rho_{vj_{l}}(\alpha(t),t)}{dt}\frac{\partial %C_{vj_{l}}(\bar{P}_{j_{l}}(t)\bar{Q}_{j_{l}}(t),\rho_{vj_{l}}(\alpha(t),t))}{\partial\rho_{vj_{l}}(\alpha(t),t)}\right)
%\nonumber\\
&=&\sum_{v\in{\cal V}}\sum_{j_{l}\in{\cal M}(i,v,t)\bigcap{\cal J}(v)}\Bigg(\rho_{vj_{l}}(\alpha(t),t)-\Lambda_{vj_{l}}(\bar{P}(t)\bar{Q}(t),\alpha(t),t)\Bigg)
\nonumber\\
&&\;\;\;\;\;\;\;\;\;\;\;\;\;\;\;\;\;\;\;\;\;\;\;\;\;\;\;\;\;\;\;\;\;\;\;\;
\frac{\partial U_{vj}(\bar{P}(t)\bar{Q}_{j_{l}}(t),\rho_{vj_{l}}(\alpha(t),t))}{\partial\rho_{vj_{l}}(\alpha(t),t)}I_{\{\bar{P}_{j_{l}}(t)\bar{Q}_{j_{l}}(t)>0\}}
\nonumber\\
&\leq&0.
\nonumber
\end{eqnarray}
Note that, the second equality in \eq{lyapunov} follows from the concavity of the utility functions and the fact that $\Lambda_{vj}(\bar{P}(t)\bar{Q}(t),\alpha(t),t)$ is the Pareto maximal Nash equilibrium policy to the utility-maximal game problem in \eq{gameoptoo} when the system is in a particular state. Thus, for any given $n\in\{0,1,2,...\}$ and each $t\in[\tau_{n},\tau_{n+1})$,
\begin{eqnarray}
\;\;\;\;\;\;\;\;0&\leq&\psi(\bar{P}(t)\bar{Q}(t),\alpha(t))
\elabel{psiine}\\
%\elabel{psiineI}\\
&\leq&\psi(\bar{P}(\tau_{n})\bar{Q}(\tau_{n}),\alpha(\tau_{n}))
\nonumber\\
&=&\sum_{v\in{\cal V}}\sum_{j_{l}\in{\cal M}(i,v,\tau_{n})\bigcap{\cal J}(v)}\frac{1}{\mu_{j_{l}}}\int_{0}^{\bar{Q}_{j_{l}}(\tau_{n})}\frac{\partial U_{vj_{l}}(\bar{P}(\tau_{n})u,\rho_{vj_{l}}(\alpha(\tau_{n})))}{\partial C_{vj_{l}}}du
\nonumber\\
&=&\sum_{v\in{\cal V}}\left(\frac{d\psi_{v}(\bar{\rho}_{vj_{1}}(\alpha(\tau_{n})))}
{dc_{vj_{1}}}\right)\left(\frac{d\psi_{v}(\bar{\rho}_{vj_{1}}(\alpha(\tau_{n-1})))}{dc_{vj_{1}}}
\right)^{-1}\psi_{v}(\bar{P}(\tau_{n})\bar{Q}(\tau_{n}), \alpha(\tau_{n-1}))
\nonumber\\
&...&
\nonumber\\
&\leq&\sum_{v\in{\cal V}}\left(\frac{d\psi_{v}(\bar{\rho}_{vj_{1}}(\alpha(\tau_{n})))}
{dc_{vj_{1}}}\right)\left(\frac{d\psi_{v}(\bar{\rho}_{vj_{1}}(\alpha(\tau_{0})))}{dc_{vj_{1}}}
\right)^{-1}\psi_{v}(\bar{P}(0)\bar{Q}(0),\alpha(0))
\nonumber\\
&\leq&\kappa\psi(\bar{P}(0)\bar{Q}(0),\alpha(0)),
\nonumber
\end{eqnarray}
where, $\kappa$ is a positive constant, i.e.,
\begin{eqnarray}
&&\kappa=\max_{v\in{\cal V}}\max_{i,j\in{\cal K}}\left(\frac{d\psi_{v}(\bar{\rho}_{vj_{1}}(i))}
{dc_{vj_{1}}}\right)\left(\frac{d\psi_{v}(\bar{\rho}_{vj_{1}}(j))}{dc_{vj_{1}}}\right)^{-1}.
\nonumber
\end{eqnarray}
Then, by the fact in \eq{psiine}, we know that $\bar{Q}(t)=0$ for all $t\geq 0$. Therefore, we complete the proof of the lemma. $\Box$
\end{proof}

\vskip 0.3cm
In the end, by considering a specific state $i\in{\cal K}$ and by the index way as used in the proof of Lemma~\ref{fluidlemma}, we can extend the proofs for Lemma 4.3 to Lemma 4.5 in Dai~\cite{dai:plamod} to the current setting. Then, by using the results in these lemmas to the proof for
Theorem 1 in \cite{dai:optrat}, we can reach a proof for Theorem~\ref{qdiff} of this paper. $\Box$

\section{Conclusion}\label{cremark}

In this paper, we make two new contributions to the future Internet modeled as IoB: developing a block based quantum channel networking technology to handle its security modeling in face of the quantum supremacy and establishing IoB based FinTech platform model with dynamic pricing for stable digital currency. The interaction between our new contributions is also addressed. In doing so, we establish a generalized IoB security model by quantum channel networking in terms of both time and space quantum entanglements with QKD. Our IoB can interact with general structured things such as supply chain, healthcare, and energy grid systems, which can be considered as a generalized IoT serving vector-valued big data streams requiring synchronized quantum computing in supporting of real-world digital economy with online trading and payment capability via stable digital currency. Thus, within our designed QKD, a generalized random number generator for private and public keys is proposed by a mixed zero-sum and non-zero-sum resource-competition oriented dynamic pricing policy. It consists of three stages: The first one is a zero-sum game-competition based decision of dynamic users' selection among all the users; The second and third ones are corresponding to a non-zero-sum game-competition based strategy for resource-sharing among selected users while conducting dynamic pricing. The effectiveness of our policy is justified by diffusion modeling with approximation theory and numerical implementations.

\end{document}